\documentclass[12pt,a4paper]{article}%
\usepackage{amssymb}
\usepackage{amsmath}
\usepackage{amsfonts}
\usepackage[dvips]{graphicx}
\usepackage{mathrsfs}
\usepackage[pdftex]{hyperref}
\usepackage{pdfpages}%
\usepackage[all]{xy}
\providecommand{\U}[1]{\protect\rule{.1in}{.1in}}

\hypersetup{
	colorlinks=true,
	citecolor = blue
}
\newtheorem{theorem}{Theorem}

\newtheorem{algorithm}[theorem]{Algorithm}

\newtheorem{definition}[theorem]{Definition}
\newtheorem{example}[theorem]{Example}

\newtheorem{lemma}[theorem]{Lemma}

\newtheorem{proposition}[theorem]{Proposition}
\newtheorem{remark}[theorem]{Remark}

\newenvironment{proof}[1][Proof]{\noindent\textbf{#1.} }{\ \rule{0.5em}{0.5em}}

\newcommand{\LD}{\mathbf{L}_\kappa^D}
\newcommand{\DomLD}{\mathscr{D}(\LD)}
\def\bL{\mathbf{L}_{\kappa}}

\usepackage{booktabs}
\usepackage{siunitx}
\usepackage{verbatim}
\usepackage{placeins}

\begin{document}
	
	\author{Abigail G. M\'arquez-Hern\'andez$^{1}$, V\'{\i}ctor A. Vicente-Ben\'{\i}{}tez$^{2}$\\{\small $^1$Divisi\'on de Ciencias e Ingenier\'ias, Universidad de Guanajuato, Campus Le\'on}\\{\small Lomas del Bosque 103, Lomas del Campestre, Le\'on, Guanajuato C.P. 37150 M\'{e}xico}\\{\small $^{2}$Instituto de Matemáticas de la U.N.A.M. Campus Juriquilla}\\{\small Boulevard Juriquilla 3001, Juriquilla, Querétaro C.P. 076230 M\'{e}xico }\\
		{\small ag.marquezhernandez@ugto.mx, va.vicentebenitez@im.unam.mx, } }
	\title{Neumann series of Bessel functions for the solutions of the Sturm-Liouville equation in impedance form and related boundary value problems}
	\date{}
	\maketitle

\begin{abstract}
		We present a Neumann series of spherical Bessel functions representation for solutions of the Sturm--Liouville equation in impedance form
\[
(\kappa(x)u')' + \lambda \kappa(x)u = 0,\quad 0 < x < L,
\]
in the case where $\kappa \in W^{1,2}(0,L)$ and has no zeros on the interval of interest.
The $x$-dependent coefficients of this representation can be constructed explicitly by means of a simple recursive integration procedure.
Moreover, we derive bounds for the truncation error, which are uniform whenever the spectral parameter $\rho=\sqrt{\lambda}$ satisfies a condition of the form $|\operatorname{Im}\rho|\leq C$.
Based on these representations, we develop a numerical method for solving spectral problems that enables the computation of eigenvalues with non-deteriorating accuracy.
	\end{abstract}
	
	\textbf{Keywords: }Spectral theory of ordinary differential equations; Sturm-Liouville problems; Neumann series of Bessel functions; Transmutation operators; Numerical analysis.
	\newline 
	
	\textbf{MSC Classification: }34A25; 34B09; 34B24; 34L16; 41A30; 47G20.

   \section{Introduction}

   The aim of this work is to construct analytical series representations for the solutions of the Sturm-Liouville equation in impedance form (SLEIF) 
   \begin{equation}\label{eq:SLEIF intro}
    -\frac{d}{dx}\left(\kappa(x)\frac{du(x)}{dx}\right)=\lambda \kappa(x) u(x),\qquad 0<x<L,\; \lambda \in \mathbb{C},
   \end{equation}
 where $\kappa \in W^{1,2}(0,L)$  is a nonvanishing function on the interval $[0,L]$, referred to as the {\it conductivity function} of the equation.
 
 Equations of the form \eqref{eq:SLEIF intro} arises in several problems of mathematical-physics, in particular, when $\kappa(x)>0$, in acoustic theory, Eq. \eqref{eq:SLEIF intro} is known as the {\it Webster horn equation} \cite{webster}, which describes the propagation of acoustic waves through a horn generated by rotating the \textit{impedance function} $a(x)=\sqrt{\kappa(x)}$. Likewise, direct and inverse eigenvalue problems of this type occur in geophysics \cite{carroll,carroll2,santosa}, wave propagation \cite{carroll1,Wu}, and classical vibration theory \cite{gladwell,vatulyan}. An interesting application appears in the determination of the shape of the human vocal tract from acoustic measurements \cite{aktosun}.

 When the conductivity function belongs to the class $W^{2,2}(0,L)$, Eq. \eqref{eq:SLEIF intro} can be transformed into a Schr\"odinger equation of the form 
 \[
    -v''(x)+\frac{a''(x)}{a(x)}v(x)=\lambda v(x),\qquad 0<x<L, \; \lambda \in \mathbb{C},
 \]
 via the Liouville transformation $v(x)=a(x)u(x)$ \cite{Everitt}. If $\kappa\in W^{1,2}(0,L)$, the resulting Schr\"odinger equation must be understood in the sense of an equation with a distributional potential \cite{mineshrodinger}.

In this work, we focus on solving direct Sturm-Liouville problems, mainly with Dirichlet and Neumann type boundary conditions. It is well known that for regular Sturm–Liouville problems, computing the eigenvalues reduces to finding the zeros of a characteristic equation, which is given by an entire function of the spectral parameter $\rho=\sqrt{\lambda}$.

Standard numerical methods for solving problems of this type include shooting and Pr\"ufer methods (see, e.g., \cite{sirca}) as well as finite difference schemes \cite{Gao}. More recent approaches for the numerical solution of direct Sturm--Liouville-type problems include, for example, the \textsc{MATSLICE} package based on \cite{matslice}, sinc--Galerkin expansions \cite{almalki}, and the method proposed in \cite{vlahakis}, where the Sturm--Liouville equation is reformulated in terms of Schwarzian derivatives. In these approaches, sufficient regularity of the coefficients is required in order to transform the problem into an equivalent Schr\"odinger equation via a Liouville transformation. However, such purely numerical methods are not easily extendable to broader classes of problems, such as those involving complex-valued coefficients or the computation of additional spectral data, including the Weyl function.

A useful tool for the effective computation of spectral data (eigenvalues and eigenfunctions) is the analytical representation of the solutions of Eq. \eqref{eq:SLEIF intro} in the form of a series of analytical functions in the parameter $\rho$. One such representation is obtained by expanding the solutions as power series in the spectral parameter $\rho$, a technique commonly known as the {\it SPPS method} \cite{sppscampos, kravchenkospp1, spps}. The coefficients in $x$ are called formal powers and can be computed through a practical recursive integration procedure \cite{spps}. In this setting, the characteristic equation is approximated by a polynomial obtained by truncating the series. This approach has been widely applied in recent years to the analysis of Sturm–Liouville-type problems \cite{sppscampos, spps}, Dirac-type systems \cite{sppsnelson}, equation pencils \cite{sppsulises}, higher-order differential equations \cite{sppsordenn}, and other related settings. However, a drawback of this type of series expansion is that the approximation in the spectral parameter is uniform only on compact sets containing the expansion point $\rho_0$. Consequently, the computation of large eigenvalues typically requires several spectral shifts of the series. This motivates the search for representations whose convergence is uniform on the entire real axis. Such representations may be constructed via integral transmutation operators. For the Schrödinger equation with a regular potential, it is known that the solutions admit an integral representation with an $L^2$ kernel \cite{marchenko}. From an expansion of this kernel in an appropriate series of orthogonal polynomials, one may derive series representations for the solutions, such as the {\it Neumann series of spherical Bessel functions (NSBF)} \cite{NSBF1}. These series have the advantageous property that the approximation in the spectral parameter is uniform on a strip $|\operatorname{Im}\rho|\leq C$. Moreover, in the case of the Schrödinger equation, they enable the computation of higher-order eigenvalues with uniform error bounds \cite{NSBF1}. For equations of the form \eqref{eq:SLEIF intro}, this expansion is typically obtained under the assumption $\kappa\in W^{2,2}(0,L)$, via the Liouville transform \cite{NSBF2}, and the existence of such series for $\kappa\in C^1[0,L]$ is also known \cite{mineimpedance1}.

The first goal of this work is to prove, in the case where $\kappa \in W^{1,2}(0,L)$, the existence of an NSBF series representation for the solution $e_{\kappa}(\rho,x)$ \eqref{eq:SLEIF intro} satisfying $e_{\kappa}^{(k)}(\rho,0)=(i\rho)^k$, $k=0,1$, in the form
\[
  e_{\kappa}(\rho,x)=e^{i\rho x}+\sum_{n=0}^{\infty}i^n\alpha_n(x)j_n(\rho x),
\]
where $j_n(z)$ denotes the spherical Bessel functions, and develop a simple integration procedure for the explicit construction of the coefficients $\{\alpha_n(x)\}_{n=0}^{\infty}$. The existence of the NSBF representation is derived through the use of an integro-differential transmutation operator associated with equation \eqref{eq:SLEIF intro}. The solution $e_{\kappa}(\rho,x)$ admits an integral representation involving a kernel $K_{\kappa}(x,t)$ of class $W^{1,2}$ defined on the triangle $0<x<L,\; -x<t<x$ \cite[Sec. 10]{mineshrodinger}. By expanding the kernel $K_{\kappa}(x,t)$ as a Fourier series in terms of Legendre polynomials, we deduce the desired NSBF representation. Due to the regularity of the transmutation kernel, we show that the remainder of the NSBF series is of order $O\left(\frac{1}{N}\right)$, uniformly for $x\in (0,L]$ and $|\operatorname{Im}\rho|\leq C$. Moreover, if the conductivity function is of class $C^p$, then the transmutation kernel is also of class $C^p$ and the order of error is of type $O\left(N^{-p-\frac{1}{2}}\right)$. We further prove that the series may be differentiated term-by-term twice, and substitution into equation \eqref{eq:SLEIF intro} yields explicit formulas for $\alpha_0, \; \alpha_1$. The remaining coefficients $\alpha_n$ can then be obtained recursively through a straightforward integration procedure.

The second part of this work is devoted to the application of the obtained NSBF representation to the solution of spectral boundary value problems. We begin by examining the main properties of the problem with Dirichlet boundary conditions and establishing key features of its spectral data, including the asymptotic behavior of the eigenvalues and the completeness of the corresponding eigenfunctions. We then construct the Weyl function associated with the Dirichlet problem using the NSBF representation. While classical analytic representations of Weyl functions typically take the form of Mittag–Leffler-type rational series involving the spectral data, such expansions often exhibit slow convergence in numerical computations. To our knowledge, no alternative analytic representations for Weyl functions are currently available, and thus NSBF series provide a new and computationally efficient approach. Using a Darboux transformation \cite{mineimpedance3}, we show that the problem with Neumann conditions for an equation with conductivity $\kappa$ is equivalent to  the Dirichlet problem with conductivity $1/\kappa$. Except for the first eigenvalue, the two problems share the same spectral data. Similar relationships are obtained for Sturm–Liouville problems with general self-adjoint boundary conditions. Finally, we illustrate the effectiveness of the proposed method for computing eigenvalues and the Weyl function through  numerical examples, including one equation with an exact characteristic equation and another with a non-smooth conductivity coefficient.

The paper is structure as follows. Section 2 presents the preliminary material concerning the Sturm–Liouville equation in impedance form and its solutions, including the SPPS and integral representations. In Section 3, we derive the Fourier series representation of the transmutation kernel and, based on this representation, we obtain the NSBF expansion for the solutions. Section 4 is devoted to the derivation and explicit construction of the NSBF coefficients. In Section 5, we study boundary value problems associated with the impedance equation. The main properties of these problems are discussed, and their solution is addressed using the NSBF representations. Finally, Section 6 presents numerical examples that demonstrate the effectiveness of the proposed method.

	\section{Background on transmutation operators for the SLEIF}
    Through the text, we use the notation $\mathbb{N}_0:=\mathbb{N}\cup\{0\}$. Let $I\subset \mathbb{R}$ be an interval. For a positive measurable  function $w$ on $I$ and $1\leq p<\infty$, we denoted by  $L_w^p(I)$ the space of measurable functions that are integrable with respect to the weight $w$. When $w\equiv 1$, we simply write $L^p(I)$. We use the standard notation $W^{k,p}(I)$ for the Sobolev space of functions in $L^p(I)$ whose first $k$ distributional derivatives belong to $L^p(I)$. Similar notation is used for function spaces on higher-dimensional domains.\\
	Let us consider the Sturm-Liouville equation in impedance form (SLEIF)
	\begin{eqnarray}\label{eq:SLEIF}
		-\frac{d}{dx}\left(\kappa(x)\frac{d u(x)}{dx}\right) = \lambda \kappa(x) u(x),\quad 0<x<L,\; \lambda \in \mathbb{C},
	\end{eqnarray}
	where the function $\kappa\in W^{1,2}(0,L)$ (in general complex-valued) does not vanish on the entire segment $[0,L]$ and satisfies the normalizing condition $\kappa(0)=1$.  The function $\kappa$ is called the {\it conductivity function}  associated with Eq. \eqref{eq:SLEIF} and its corresponding {\it impedance function} is defined by $a^2(x):=\kappa(x)$. Eq. \eqref{eq:SLEIF} can be rewritten as
	\begin{equation}\label{eq:2ndSLEIF}
		-u''(x)+q(x)u'(x)=\rho^2 u(x), \qquad \mbox{where  } q(x)=-\frac{\kappa'(x)}{\kappa(x)} \mbox{ and } \rho^2=\lambda.
	\end{equation}
	Let $e_{\kappa}(\rho,x)$ be the unique solution of \eqref{eq:SLEIF} satisfying the initial conditions
	\begin{equation}\label{eq:initialconde}
		e_{\kappa}(\rho,0)=1\;\; \mbox{ and }\; \;e'_{\kappa}(\rho,0)=i\rho.
	\end{equation}
	The solution $e_{\kappa}(\rho,x)$ admits some analytical representations. The first one is given in terms of the so-called spectral parameter power series (SPPS). Define the system of functions $\{\varphi_{\kappa}^{(k)}\}_{k=0}^{\infty}$ (called the {\it formal powers} associated with $\kappa$) as follows:
	\begin{align}
		\varphi_{\kappa}^{(0)}(x) = & 1,\label{eq:1stformalpower}\\
		\varphi_{\kappa}^{(1)}(x) = & \int_0^x\frac{dt}{\kappa(t)},\label{eq:2ndformalpower}\\
		\varphi_{\kappa}^{(k)}(x) = & k(k-1)\int_0^x\frac{1}{\kappa(t)}\left[\int_0^t\kappa(s)\varphi_{\kappa}^{(k-2)}(s)ds\right]dt,\quad  k\geq 2. \label{eq:nthformalpower}
	\end{align}
	
	\begin{theorem}[\cite{sppscampos,spps}]\label{th:spps}
		The solution $e_{\kappa}(\rho,x)$ admits the SPPS representation
		\begin{equation}\label{eq:SPPSseries}
			e_{\kappa}(\rho,x)=\sum_{k=0}^{\infty}\frac{(i\rho)^k\varphi_{\kappa}^{(k)}(x)}{k!}.
		\end{equation}
		This series converges uniformly and absolutely with respect to $x$ on $[0,L]$, as does the series obtained by termwise differentiation once. The series of second derivatives converges in $L^2(0,L)$. The series converges uniformly and absolutely with respect to $\rho$ on compact subsets of the complex plane.
	\end{theorem}
	It is known that the solution $e_{\kappa}(\rho,x)$ admits an integral representation involving an $L^2$ kernel that preserves certain smoothness properties of the conductivity $\kappa$. The following theorem summarizes results presented in \cite[sec. 3]{mineimpedance1}, \cite[Sec. 5 and 6]{mineimpedance3} and \cite[Th. 40]{mineshrodinger}

     \begin{theorem}
         For every $x\in (0,L]$,  the solution $e_{\kappa}(\rho,x)$ admits the integral representation 
         \begin{equation}\label{eq:integralrep}
             e_{\kappa}(\rho,x)=e^{i\rho x}-i\rho \int_{-x}^xK_{\kappa}(x,t)e^{i\rho t}dt,
         \end{equation}
         where the kernel $K_{\kappa}(x,t)$ belongs to $W^{1,2}(\mathcal{T}_L)\cap C(\overline{\mathcal{T}_L})$, with $\mathcal{T}_L:=\{(x,t)\in \mathbb{R}^2\,|\, 0< x< L, |t|< x\}$ (see Figure \ref{fig:triangle}), and satisfies the Goursat conditions 
         \begin{equation}\label{eq:goursatcond}
             K_{\kappa}(x,x)=1-\frac{1}{\sqrt{\kappa(x)}},\quad K_{\kappa}(x,-x)=0.
         \end{equation}
         Moreover, if $\kappa\in C^p[0,L]$ with $p\in \{1,2\}$, hence $K_{\kappa}\in C^p(\overline{\mathcal{T}_L})$, 
     \end{theorem}
    \begin{figure}[h!]
    \centering
    \includegraphics[width=5.5cm]{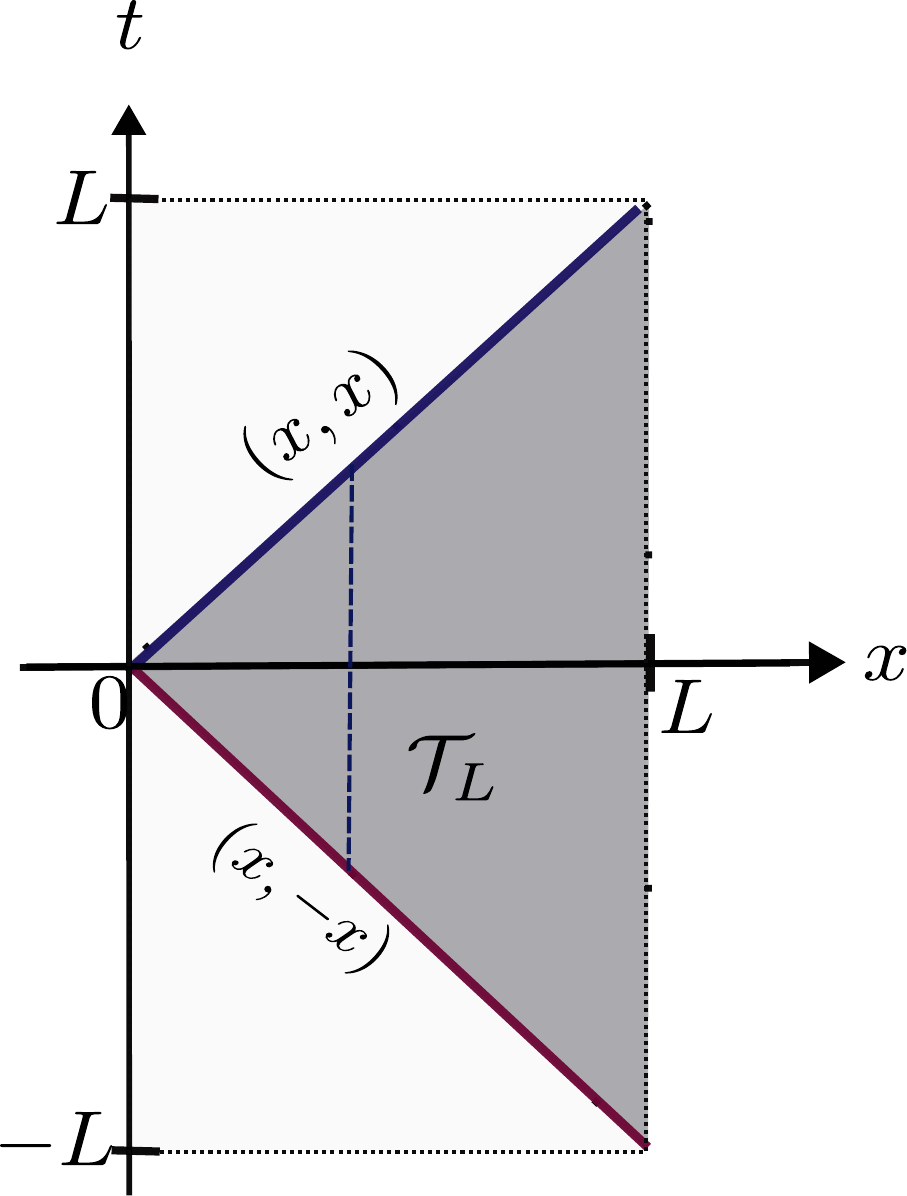}
    \caption{The domain $\mathcal{T}_L$. The domain of integration $-x<t<x$ is indicated by the blue dotted line.}
    \label{fig:triangle}
\end{figure}
	The kernel $K_{\kappa}$ is called the {\it canonical transmutation} kernel for Eq. \eqref{eq:SLEIF}. Define the operator
	\begin{equation}\label{eq:transmop}
		\mathbf{T}_{\kappa}u(x):= u(x)-\int_{-x}^{x}K_{\kappa}(x,t)u'(t)dt,\quad u\in C^1[-L,L].
	\end{equation} 
    Hence, relation \eqref{eq:integralrep} can be written as
    \begin{equation*}
        e_{\kappa}(\rho,x)=\mathbf{T}_{\kappa}[e^{i\rho x}].
    \end{equation*}
    The operator $\mathbf{T}_{\kappa}$ is called the {\it canonical transmutation operator} for Eq. \eqref{eq:SLEIF}  \cite{mineimpedance1,mineimpedance3,mineshrodinger}.

    \begin{theorem}[\cite{mineimpedance3,mineshrodinger}]\label{Th:transmutationproperty}
	The operator $\mathbf{T}_{\kappa}:W^{1,2}(-L,L)\rightarrow W^{1,2}(0,L)$ is bounded and satisfies that if $v\in W^{3,2}(-L,L)$ is a solution of $v''+\rho^2v=0$ in $(-L,L)$, then $u=\mathbf{T}_{\kappa}v$ is a solution of Eq. \eqref{eq:SLEIF} satisfying the initial conditions $u(0)=v(0)$ and  $u'(0)=v'(0)$. Furthermore, the following transmutation property holds:
		\begin{equation}\label{eq:transprop}
			\mathbf{T}_{\kappa}[x^k]=\varphi_{\kappa}^{(k)}(x),\qquad \forall k\in \mathbb{N}_0.
		\end{equation}
	\end{theorem}

\section{Neumann series of Bessel functions representation}

\subsection{The Fourier-Legendre representation for the kernel}
	Let $\{P_n(z)\}_{n=0}^{\infty}$ be the Legendre polynomials in $L^2(-1,1)$ with norm $\|P_n\|_{L^2(-1,1)}=\sqrt{\frac{2}{2n+1}}$, $n\in \mathbb{N}_0$. It is well known that $\{P_n(z)\}_{n=0}^{\infty}$ is an orthogonal basis for $L^2(-1,1)$. For a fixed $x\in (0,L]$, $K_{\kappa}(x,\cdot)\in L^2(-x,x)$, and therefore it admits a Fourier expansion in terms of Legendre polynomials:
\begin{equation}\label{eq:FourierLegendre}
	K_{\kappa}(x,t)=\sum_{n=0}^{\infty}\frac{a_n(x)}{x}P_n\left(\frac{t}{x}\right).
\end{equation}
The series converges with respect to $t$ in $L^2(-x,x)$, and the Fourier-Legendre coefficients are given by
\begin{equation}\label{eq:FLcoeff}
	a_n(x) = \left(n+\frac{1}{2}\right)\int_{-x}^{x}K_{\kappa}(x,t)P_n\left(\frac{t}{x}\right)dt.
\end{equation}
Writting $P_n(z)=\sum_{k=0}^{n}l_{k,n}z^k$, we obtain a representation for the coefficients $\{a_n(x)\}_{n=0}^{\infty}$ in terms of the formal powers.
\begin{proposition}
	The functions $\{a_n\}_{n=0}^{\infty}$ belong to $W^{2,2}(0,L)$ and satisfy the relations
	\begin{equation}\label{eq:alphaintermsofphi}
		a_n(x)= \left(n+\frac{1}{2}\right)\sum_{k=0}^{n}\frac{l_{k,n}}{(k+1)}\left(\frac{x^{k+1}-\varphi_{\kappa}^{(k+1)}(x)}{x^k}\right)\qquad \forall n\in \mathbb{N}_0.
	\end{equation}
\end{proposition}
\begin{proof}
    The argument is identical to the proof of \cite[Prop.~38]{mineimpedance1}, replacing the use of \cite[Th.~37]{mineimpedance1} with \cite[Th.~37]{mineimpedance3}, together with the transmutation property \eqref{eq:transprop}.
\end{proof}

\begin{remark}
    Given the characteristic condition \eqref{eq:goursatcond}
    \[
     K_{\kappa}(x,x)=1-\frac{1}{\sqrt{\kappa(x)}},
    \]
    using the Fourier-Legendre series representation \eqref{eq:FourierLegendre} together with the fact that $P_n=\left(\frac{x}{x}\right)=P_n(1)=1$ for every $n\in \mathbb{N}_0$ (see \cite[Formula (17) from p. 124]{suetin}), we obtain the identity
    \begin{equation}
    1-\frac{1}{\sqrt{\kappa(x)}}= \sum_{n=0}^\infty \frac{a_n(x)}{x}.
        \label{eq:goursatseries}
    \end{equation}
\end{remark}
Denote the $N$-th partial sum of \eqref{eq:FourierLegendre} by $K_{\kappa,N}(x,t):=\sum_{n=0}^N\frac{a_n(x)}{x}P_n\left(\frac{t}{x}\right)$.
\begin{proposition}[\cite{mineimpedance1}, Prop. 37]\label{Prop.convergenceoffouriserseries}
    Let $p\in \{0,1\}$. If $\kappa\in C^{p+1}[0,L]$, then there exists a constant $c_p>0$ such that the following inequality holds
    \begin{equation}\label{eq:remainderfourer}
        \sup_{\overset{0<x\leq L}{-x\leq t\leq x}}|K_{\kappa}(x,t)-K_{\kappa,N}(x,t)|\leq \frac{c_pM_pL^{p+1}}{N^{p+\frac{1}{2}}}\quad \forall N>p,
    \end{equation}
    where $M_p:=\left\|\frac{\partial^{p+1}K_{\kappa}}{\partial t^{p+1}}\right\|_{C(\overline{\mathcal{T}_L})}$.
\end{proposition}

Let $g\in L^2(-1,1)$ and set $\widehat{g}_n=\left(n+\frac{1}{2}\right)\langle g, P_n\rangle_{L^2(-1,1)}$ for $n\in \mathbb{N}_0$. Let $\mathcal{P}_N[-1,1]$ denote the set of polynomial functions of degree at most $N$ on $[-1,1]$ and set$E_N^2(f)=\inf_{p\in \mathcal{P}_N[-1,1]}\|g-p\|_{L^2(-1,1)}$. In this case, $E_N^2(f)=\|g-g_N\|_{L^2(-1,1)}$, where $g_N=\sum_{n=0}^N\widehat{g}_nP_n$.

The following lemma describes the decay rate of the Fourier–Legendre coefficients $\widehat{g}_n$ and of the error $E_N^2(g)$ when $g\in W^{1,2}(-1,1)$.

\begin{lemma}\label{lemma:aproxfourier1}
    There exist positive constants $c_0$ and $c_1$ such that for every $g\in W^{1,2}(-1,1)$, the following estimates hold:
    \begin{itemize}
        \item[(i)] $E_N^2(g)\leq \frac{c_0}{N}\|g'\|_{L^2(-1,1)}$ for all $N\in \mathbb{N}$. 
        \item[(ii)] $|\widehat{g}_n|\leq \frac{c_1}{\sqrt{n}}\|g'\|_{L^2(-1,1)}$ for all $n\in \mathbb{N}$.
    \end{itemize}
\end{lemma}
\begin{proof}
    For (i), according to \cite[Th. 1]{nxky}, there exists a constant $c_0>0$, independent on $N$ and $g$,  such that $E_N^2(g)\leq \frac{c_0}{N}\|\Delta\cdot  g'\|_{L^2(-1,1)}$, where $\Delta(z):=\sqrt{1-z^2}$, $1\leq z\leq 1$. Since $\|\Delta \cdot g'\|_{L^2(-1,1)}\leq \|g'\|_{L^2(-1,1)}$, estimate (i) follows.

    For (ii), we use the recurrence relation \cite[Formula (32) from p. 126]{suetin}
    \[
       P_n=\frac{1}{2n+1}\left(P'_{n+1}-P'_{n-1}\right) \quad \text{ for } n\geq 1,
    \]
    together with the boundary values $P_n(\pm1)=(\pm1)^n$ \cite[Formula (17) from p. 124]{suetin}. Integration by parts in $\widehat{g}_n$ yields
    \begin{align*}
        \widehat{g}_n&= \left(n+\frac{1}{2}\right)\int_{-1}^{1}f\left(\frac{P'_{n+1}-P'_{n-1}}{2n+1}\right)= \frac{1}{2}\left(\int_{-1}^1f'P_{n-1}-\int_{-1}^1f'P_{n+1}\right).
    \end{align*}
    Applying the Cauchy-Bounyakovsky-Scharz inequality, we get
    \begin{align*}
        |\widehat{g}_n|\leq \frac{1}{2} \|f'\|_{L^2(-1,1)}\left(\sqrt{\frac{2}{2n-1}}+\sqrt{\frac{2}{2n+3}}\right)=\frac{1}{\sqrt{2}}\|f'\|_{L^2(-1,1)}\left(\frac{1}{\sqrt{n-\frac{1}{2}}}+\frac{1}{\sqrt{n+\frac{3}{2}}}\right).
    \end{align*}
    Since $\displaystyle \lim_{n\rightarrow \infty}\frac{\sqrt{n}}{\sqrt{n+1\pm \frac{1}{2}}}=1$, it follows that there is a constant $C>0$ such that $\frac{1}{\sqrt{n+1\pm \frac{1}{2}}}\leq \frac{C}{\sqrt{n}}$ for $n\geq 1$. Taking $c_1=\frac{C}{\sqrt{2}}$, we obtain (ii).
\end{proof}\\
From this lemma, we obtain estimates for the Fourier-Legendre coefficients and the error for the kernel $K_{\kappa}$ in the general case when $\kappa\in W^{1,2}(0,L)$.

\begin{theorem}\label{Th:approximationkernel}
    Let $\kappa\in W^{1,2}(0,L)$. Define $d_{\kappa}:=\left\|\frac{\kappa'}{2\kappa}\right\|_{L^2(0,\ell)}$ and set
    \begin{equation}\label{eq:constantfornorm}
        M_{\kappa,L}:= 8L\|\kappa^{-1/2}\|_{L^{\infty}(0,L)} (L d_{\kappa}+2d_{\kappa}^2(L^2d_{\kappa}^2+L)e^{L d_{\kappa}}).
    \end{equation}
    Let $c_0$ and $c_1$ be the constants from Lemma \ref{lemma:aproxfourier1}. The following estimates hold:
    \begin{itemize}
        \item[(i)] $\displaystyle \sup_{0<x\leq L}\|K_{\kappa}(x,\cdot)- K_{\kappa,N}(x,\cdot)\|_{L^2(-x,x)}\leq \frac{c_0\sqrt{LM_{\kappa,L}}}{N}$ for $N\in \mathbb{N}$. 
        \item[(ii)] $\displaystyle \sup_{0<x\leq L}|a_n(x)|\leq \frac{c_1\sqrt{LM_{\kappa,L}}}{\sqrt{n}}$ for $n\in \mathbb{N}$.
    \end{itemize}
\end{theorem}

\begin{proof}
    Fix $x\in (0,\ell]$, and consider $g_x(z)=K_{\kappa}(x,xz)$, $-1\leq z\leq 1$. Hence $g_x\in W^{1,2}(-1,1)$ with $g_x'(z)=x\left(\frac{\partial K_{\kappa}}{\partial t}\right)(x,xz)$. Changing variables in formula \eqref{eq:FLcoeff} shows that $a_n(x)=x\widehat{(g_x)}_n$ and that \eqref{eq:FourierLegendre} is precisely the Fourier-Legendre series of $g_x$. Now, according to \cite[Remark 39]{mineshrodinger},
    \[
        \sup_{0<x\leq L}\int_{-x}^x\left|\frac{\partial K_{\kappa}(x,t)}{\partial t}\right|^2dt\leq M_{\kappa,L}.
    \]
    Hence
    \[
    \|g_x'\|_{L^2(-1,1)}^2=x^2\int_{-1}^1\left|\left(\frac{\partial K_{\kappa}}{\partial t}\right)(x,xz)\right|^2dx=x\int_{-x}^x\left|\frac{\partial K_{\kappa}(x,t)}{\partial t}\right|^2dt\leq LM_{\kappa,L}.
    \]
    Thus, the result follows from Lemma \ref{lemma:aproxfourier1}.
\end{proof}

\subsection{NSBF representation for the solutions}
    
Substituting the Fourier-Legendre series \eqref{eq:FourierLegendre} into the integral representation \eqref{eq:integralrep} yields
	\begin{align*}
		e(\rho,x) = & e^{i\rho x}-i\rho \sum_{n=0}^{\infty}\frac{a_n(x)}{x}\int_{-x}^{x}P_n\left(\frac{t}{x}\right)e^{i\rho t}dt.
	\end{align*}
	The interchange of the series with the integral is due to the convergence in the variable $t$ with respect to the norm of $L^2(-x,x)$. Observe that
	\begin{align*}
		\int_{-x}^{x}P_n\left(\frac{t}{x}\right)e^{i\rho t}dt = \int_{-x}^{x}P_n\left(\frac{t}{x}\right)\cos(\rho t)dt+i\int_{-x}^{x}P_n\left(\frac{t}{x}\right)\sin(\rho t)dt,
	\end{align*}
	since $P_{2n}(z)$ and $P_{2n+1}(z)$ are even and odd functions, respectively, we obtain
	\[
	\int_{-x}^{x}P_n\left(\frac{t}{x}\right)e^{i\rho t}dt =\begin{cases}
		2\int_0^xP_{n}\left(\frac{t}{x}\right)\cos(\rho t)dt,& \text{ if } n \text{ is even},\\
		2i\int_0^xP_n\left(\frac{t}{x}\right)\sin(\rho t)dt, & \text{ if } n \text{ is odd},
		
	\end{cases}
	\]
	and using the formulas 2.17.7 from \cite[p. 433]{pruvnikov},
	\[
	\int_{0}^{a}\left\{\begin{matrix}
		P_{2k}\left(\frac{t}{a}\right)\cos(bt)\\
		P_{2k+1}\left(\frac{t}{a}\right)\sin(bt)
	\end{matrix}\right\}dt= (-1)^k\sqrt{\frac{\pi a}{2b}}J_{2k+\delta+\frac{1}{2}}(ab), \; \delta=\left\{\begin{matrix}0\\1\end{matrix}\right\}, \;a>0,
	\]
	where $J_{\nu}(z)$ stands for the Bessel functions of the first kind. If we consider the the spherical Bessel functions $j_{\nu}(z):= \sqrt{\frac{\pi}{2z}}J_{\nu+\frac{1}{2}}(z)$, we get the formulas
	\[
	\int_{0}^{a}\left\{\begin{matrix}
		P_{2k}\left(\frac{t}{b}\right)\cos(bt)\\
		P_{2k+1}\left(\frac{t}{b}\right)\sin(bt)
	\end{matrix}\right\}dt= (-1)^kaj_{2k+\delta}(ab), \; \delta=\left\{\begin{matrix}0\\1\end{matrix}\right\},\; a>0.
	\]
	Again, by the parity of $P_{n}(z)$, we obtain
	\[
	\int_{-x}^{x}P_n\left(\frac{t}{x}\right)e^{i\rho t}dt=2xi^nj_n(\rho x).
	\]
	Setting $\alpha_n(x)=2a_n(x)$, we obtain the series representation
	\begin{equation}\label{eq:NSBFe}
		e(\rho,x)=e^{i\rho x}-i\rho \sum_{n=0}^{\infty}i^n\alpha_n(x)j_n(\rho x).
	\end{equation}
	
	\begin{proposition}
		For $x\in (0,L]$ and $\rho\in \mathbb{C}$, the solution $e_{\kappa}(\rho,x)$ admits the series representation \eqref{eq:NSBFe}. The series converges pointwise in $x$, and uniformly in $\rho$ on compact subset of $\mathbb{C}$. Furthermore, for $N\in \mathbb{N}$ and $x\in (0,L)$ fixed, if we set
		\begin{equation}\label{eq:NSBFseries1}
			e_{\kappa,N}(\rho,x):=e^{i\rho x}-i\rho \sum_{n=0}^{N}i^n\alpha_n(x)j_n(\rho x),
		\end{equation}
		then for any $\rho\in \mathbb{C}$ such that $|\operatorname{Im}\rho|\leq C$, $C>0$, we have the estimate
		\begin{equation}\label{eq:estimateNSBF}
			\left|\frac{1}{\rho}( e_{\kappa}(\rho,x)-e_{\kappa,N}(\rho,x))\right|\leq \varepsilon_N(x)\sqrt{\frac{\sinh(2LC)}{LC}},
		\end{equation}
		where 
		\begin{equation}\label{eq:estimatepsN}
			\varepsilon_N(x):=\|K_{\kappa}(x,\cdot)-K_{k,N}(x,\cdot)\|_{L_2(-x,x)}\leq \frac{c_0\sqrt{LM_{\kappa,L}}}{N}, 
		\end{equation}
        with $c_0$ is as in Lemma \ref{lemma:aproxfourier1} and $M_{\kappa,L}$ as in Theorem \ref{Th:approximationkernel}.

        Moreover, if $\kappa\in C^{p+1}[0,L]$ for some $p\in \{0,1\}$, then
        \begin{equation}\label{eq:remaindernsbf}
        \varepsilon_N(x)\leq \frac{2C_pM_pL^{p+2}}{N^{p+\frac{1}{2}}}\qquad \forall x\in (0,L], \; N>p,
        \end{equation}
        where $C_p$ and $M_p$ are given as in Proposition \ref{Prop.convergenceoffouriserseries}.
	\end{proposition}
	\begin{proof}
		It remains to prove estimate \eqref{eq:estimateNSBF}. Set $K_{\kappa,N}(x,t)=\sum_{n=0}^{N}\frac{a_n(x)}{x}P_n\left(\frac{t}{x}\right)$. It is clear that 
		\[
		e_{\kappa,N}(\rho,x)=e^{i\rho x}-i\rho \int_{-x}^{x}K_{\kappa,N}(x,t)e^{i\rho t}dt.
		\]
		Consequently,
		\begin{equation*}
			\frac{1}{\rho}(e_{\kappa}(\rho,x)-e_{\kappa,N}(\rho,x))=-i\int_{-x}^{x}(K_{\kappa}(x,t)-K_{\kappa,N}(x,t))e^{i\rho t}dt,
		\end{equation*} 
		and the singularity at $\rho=0$ is removable. Applying the Cauchy-Bunyakovsky-Schwarz inequality yields
		\begin{align*}
			\left| \frac{1}{\rho}(e(\rho,x)-e_N(\rho,x))\right|  &\leq \left(\int_{-x}^{x}|K_{\kappa}(x,t)-K_{\kappa,N}(x,t)|^2dt\right)^{\frac{1}{2}}\left(\int_{-x}^{x}|e^{i\rho t}|^2dt\right)^{\frac{1}{2}}.
		\end{align*} 
		Set $\varepsilon_N(x)=\left(\int_{-x}^{x}|K_{\kappa}(x,t)-K_{\kappa,N}(x,t)|^2dt\right)^{\frac{1}{2}}=\|K_{\kappa}(x,\cdot)-K_{\kappa,N}(x,\cdot)\|_{L_2(-x,x)}$, and note that
		\begin{align*}
			\int_{-x}^{x}|e^{i\rho t}|^2dt=\int_{-x}^{x}e^{-2\operatorname{Im}\rho}dt= \frac{\sinh(2\operatorname{Im}\rho x)}{\operatorname{Im}\rho x}.
		\end{align*}
		Since $\frac{\sinh(2\xi)}{\xi}$ is even and increasing for $\xi\geq 0$, and because $|\operatorname{Im}\rho x|\leq LC$, we obtain \eqref{eq:estimateNSBF}. Finally, inequalities \eqref{eq:estimatepsN} and \eqref{eq:remaindernsbf} follow from Theorem \ref{Th:approximationkernel}(i) and Proposition \ref{Prop.convergenceoffouriserseries}, respectively.
	\end{proof}\\
	Note that \eqref{eq:estimateNSBF} implies that the approximation is uniform on $x\in (0,L]$ and for all $\rho$ belonging to the strip $|\operatorname{Im}\rho|\leq C$. In particular, the approximation is uniform for all $\rho \in \mathbb{R}$.\\
Let us consider the solutions $C_{\kappa}(\rho,x)$ and $S_{\kappa}(\rho,x)$ of \eqref{eq:SLEIF} satisfying the initial conditions
\begin{align}
	C_{\kappa}(\rho,0)=1,  & \;\; C'_{\kappa}(\rho,0)=0, \label{eq:conditionscosine}\\
	S_{\kappa}(\rho,0)=0, & \;\;S'_{\kappa}(\rho,0)=1. \label{eq:conditionssine}
\end{align}
Thus, ${C_\kappa(\rho,x),S_\kappa (\rho,x)}$ forms fundamental set of solutions for Eq.\eqref{eq:SLEIF}. A direct computation shows that
	\begin{equation}\label{eq:definitionsolutioncosineandsine}
	C_{\kappa}(\rho,x):=\frac{e_{\kappa}(\rho,x)+e_{\kappa}(-\rho,x)}{2},\quad S_{\kappa}(\rho,x):= \frac{e_{\kappa}(\rho,x)-e_{\kappa}(-\rho,x)}{2i\rho}.
\end{equation}
Using  \eqref{eq:definitionsolutioncosineandsine} together with the parity property of the spherical Bessel functions $j_n(-z)=(-1)^n j_n(z)$ (see \cite[Formula (10.1.3)]{abramovitz}), we deduce the following NSBF representations for the solutions $C_\kappa(\rho,x)$ and $S_\kappa(\rho,x)$. 

\begin{proposition}
For $x\in (0,L]$ and $\rho\in \mathbb{C}$, the solutions $C_\kappa(\rho,x)$ and $S_\kappa(\rho,x)$ admit the following NSBF representations
    \begin{align}
    C(\rho,x) =& \cos(\rho x) \;+\; \rho \sum_{k=0}^{ \infty } (-1)^k \, \alpha_{2k+1}(x) \, j_{2k+1}(\rho x), \label{eq:NSBFcoseno}\\
    S(\rho,x) =& \frac{\sin(\rho x)}{\rho} \;+\; \sum_{k=0}^{\infty } (-1)^{k+1} \, \alpha_{2k}(x) \, j_{2k}(\rho x).\label{eq:NSBFseno}
    \end{align}
The series converges uniformly on $x\in (0,L]$, and uniformly in $\rho$ on compact subset of $\mathbb{C}$.
\end{proposition}
The reminder for the partial sums of series \eqref{eq:NSBFcoseno} and \eqref{eq:NSBFseno} satisfy estimates similar to  \eqref{eq:estimateNSBF} and \eqref{eq:estimatepsN}.
    \begin{remark}\label{remark:firstalpha}
        From relations \eqref{eq:alphaintermsofphi} we obtain the first coefficients
        \begin{align}
            \alpha_0(x) &=x-\varphi_{\kappa}^{(1)}(x),\label{eq:alpha0}\\
            \alpha_1(x) &= \frac{3x}{2}-\frac{3}{2x}\varphi_{\kappa}^{(2)}(x). \label{eq:alpha1}
        \end{align}
    \end{remark}
The following lemma will be useful for the analysis of the differentiability of the NSBF series.
\begin{lemma}\label{Lemma:convergencebessel}
\begin{itemize}
    \item[(i)] Let $Q_k(\zeta)$ be a polynomial of degree $k\in \mathbb{N}$. For any $z\in \mathbb{C}$, the series
    \[
      \sum_{n=0}^{\infty}|Q_k(n)||j_n(z)|^2
    \]
    converges uniformly on compact subsets of $\mathbb{C}$.
    \item[(ii)] Let $f\in L^2(-1,1)$ and  $\widehat{f}_n:=\left(n+\frac{1}{2}\right)\langle f, P_n\rangle_{L^2(-1,1)}$, $n\in \mathbb{N}$. Then the series $\sum_{n=0}^{\infty}\widehat{f}_nj_n(z)$ satisfies the estimate
    \[
    \sum_{n=0}^{\infty}|\widehat{f}_n||j_n(z)|\leq \|f\|_{L^2(-1,1)}\left(\sum_{n=0}^{\infty}\frac{2n+2}{2}|j_n(z)|^2\right)^{\frac{1}{2}},
    \]
    and converges absolutely and uniformly on compact subsets of $\mathbb{C}$.
\end{itemize}   
\end{lemma}
\begin{proof}
    \begin{itemize}
        \item[(i)] Let us consider the estimate \cite[Eq. 9.1.62]{abramovitz}
        \[
        |j_n(z)|\leq \sqrt{\pi}\left|\frac{z}{2}\right|^n\frac{e^{|\operatorname{Im}z|}}{\Gamma\left(n+\frac{3}{2}\right)}.
        \]
        Choose constants $C_k>0$ and $N\in \mathbb{N}$ such that $|Q_k(z)|\leq C_k|z|^k$ for $|z|\geq N$. Recall that $\Gamma\left(n+\frac{3}{2}\right)=\left(n+\frac{1}{2}\right)\Gamma\left(n+\frac{1}{2}\right)=\left(n+\frac{1}{2}\right)\frac{(2n-1)!!}{2^n}\sqrt{\pi}$. Hence
    \begin{align*}
        \sum_{n=N}^{\infty}|Q_k(n)||j_n(z)|^2 &\leq \sum_{n=N}^{\infty}\frac{C_kn^k}{2^{2n}}\cdot\frac{e^{2|\pi |z|^{2n}\operatorname{Im}z|} }{\frac{\left(n+\frac{1}{2}\right)^2((2n-1)!!)^2\pi}{2^{2n} }}=C_ke^{2|\operatorname{Im}z|}\sum_{n=N}^{\infty}\frac{n^k2|z|^{2n}}{\left(n+\frac{1}{2}\right)^ 2((2n-1)!!)^2}. 
    \end{align*}
    The ratio test shows the convergence of this series, and hence by the Weierstrass M-test it converges uniformly on compact subsets of $\mathbb{C}$.
    \item[(ii)] The Fourier-Legendre series of $f$ in terms of the normalized Legendre polynomials $\widehat{P}_n=\sqrt{n+\frac{1}{2}}P_n$, $n\in \mathbb{N}_0$, is $f=\sum_{n=0}^{\infty}\left(\frac{\widehat{f}_n}{\sqrt{n+\frac{1}{2}}}\right)\widehat{P}_n$. By Parseval's identity $\|f\|_{L^2(-1,1)}^2=\sum_{n=0}^{\infty}\frac{2|f_n|^2}{2n+1}$. Using the Cauchy-Bunyakovsky-Schwarz inequality, we obtain
    \begin{align*}
        \sum_{n=0}^{\infty}|\widehat{f}_n||j_n(z)| &\leq \left(\sum_{n=0}^{\infty}\frac{2|\widehat{f}_n|^2}{2n+1}\right)^{\frac{1}{2}}\left(\sum_{n=0}^{\infty}\frac{2n+1}{2}|j_n(z)|^2\right)^{\frac{1}{2}}=\|f\|_{L^2(-1,1)}\left(\sum_{n=0}^{\infty}\frac{2n+1}{2}|j_n(z)|^2\right)^{\frac{1}{2}}.
    \end{align*}
    By the point (i), the series converges absolutely and uniformly on compact subsets of $\mathbb{C}$. 
    \end{itemize}
    
\end{proof}

\section{Recursive relations for the NSBF coefficients}	
	In this section, we develop an alternative recursive integration procedure for computing the NSBF coefficients $\{\alpha_n(x)\}_{n=0}^{\infty}$.
	
	\begin{theorem}
		For $\kappa\in W^{1,2}(0,L)$, the Fourier-Legendre coefficients $\{\alpha_m(x)\}_{m=0}^{\infty}$ satisfy the following recursive relations: $\alpha_m(x)=\frac{\sigma_m(x)}{x^m}$, $m\in \mathbb{N}_0$, where
		\begin{align}
			\sigma_0(x) = & x-\int_0^x\frac{dt}{\kappa(t)}, \label{eq:sigma0}\\
			\sigma_1(x) = & \frac{3}{2}x^2-3\int_0^x\frac{dt}{\kappa(t)}\int_0^t\kappa(s)ds,\label{eq:sigma1} \\
			\theta_m(t):= & \frac{1}{\kappa(t)}\int_0^t\left\{ 2(m-1)\kappa(s)+s\kappa'(s)\right\}\sigma_{m-2}(s)ds,\nonumber\\
			\eta_m(x) := & \int_0^x\left\{2(2m-1)t\sigma_{m-2}(t)-(2m-1)\theta_m(t)\right\}dt,\nonumber\\
			\sigma_m(x) = & \frac{2m+1}{2m-3}\left[x^2\sigma_{m-2}(x)-\eta_m(x)\right], \; \; m\geq 2. \label{eq:simgamgeq2}
		\end{align}
	\end{theorem}

\begin{proof}
    First, we assume that the series \eqref{eq:NSBFseries1} can be differentiated term by term twice. Since $e_{\kappa}(\rho,x)$ is a solution of \eqref{eq:2ndSLEIF}, we get
    \begin{align}
        0=&\sum_{n=0}^{\infty}i^n\left\{\alpha_n''j_n(\rho x)+2\alpha_n'j_n'(\rho x)\rho +\alpha_n(\rho^2j_n''(\rho x)+\rho^2j_n(\rho x)\right\}\nonumber\\
        &\quad +q(x)e^{i\rho x}-q(x)\sum_{n=0}^{\infty}i^n\left\{\alpha_n'j_n(\rho x)+\alpha_n\rho j_n'(\rho x)\right\}. \label{eq:auxiliar1}
    \end{align}
   We employ the recursive relation \cite[Formula 10.1.22]{abramovitz}
   \begin{equation}\label{eq:auxiliar2}
       j_k'(z)=-j_{k+1}(z)+\frac{k}{z}j_k(z),\quad k\in \mathbb{N}_0.
   \end{equation}
   To compute the second derivative, we differentiate \eqref{eq:auxiliar2}, obtaining
   \[
   j_k''(z)=-j_{k+1}'(z)-\frac{k}{z^2}j_k(z)+\frac{k}{z}j_k'(z),
   \]
   and the recursive relation \cite[Formula 10.1.21]{abramovitz}
   \[
   j_k'(z)=j_{k-1}(z)-\frac{k+1}{z}j_k(z),\quad k\in \mathbb{N},
   \]
   we obtain
   \begin{equation}\label{eq:auxiliar3}
       j_k''(z)=-j_k(z)+\frac{2}{z}j_{k+1}(z)+\frac{k(k-1)}{z^2}j_k(z),\quad k\in \mathbb{N}_0.
   \end{equation}
   From these relations, we derive 
   \begin{align}
       \rho j_n'(\rho x) &= -\rho j_{n+1}(\rho x)+\frac{n}{x}j_n(\rho x), \label{eq:besselaux1}\\
       \rho^2j_n''(\rho x) &=-\rho^2j_n(\rho x)+\frac{2\rho}{x}j_{n+1}(\rho x)+\frac{n(n-1)}{x^2}j_n(\rho x). \label{eq:besselaux2} 
   \end{align}
 Substituting \eqref{eq:besselaux1} and \eqref{eq:besselaux2} into \eqref{eq:auxiliar1}, we obtain
   \begin{align*}
       0=& \sum_{n=0}^{\infty}i^n\left\{\alpha_n''j_n(\rho x)+2\alpha_n'\left[-\rho j_{n+1}(\rho x)+\frac{n}{x}j_n(\rho x)\right]+\alpha_n\left[\frac{2\rho}{x}j_{n+1}(\rho x)+\frac{n(n-1)}{x^2}j_n(\rho x)\right]\right\}\\
       &\; -\sum_{n=0}^{\infty}q(x)i^n\left\{\alpha_n'j_n(\rho x)+\alpha_n\left[-\rho j_{n+1}(\rho x)+\frac{n}{x}j_n(\rho x)\right]\right\}+q(x)e^{i\rho x}.
   \end{align*}
Since $e^{i\rho x}=j_0(\rho x)-\rho xj_1(\rho x)+i\rho xj_0(\rho x)$, hence   
\begin{align*}
       0=& \sum_{n=1}^{\infty}i^n\left\{\alpha_n''+\frac{2n}{x}\alpha_n'+\frac{n(n-1)}{x^2}\alpha_n-q(x)\alpha_n'-q(x)\frac{n}{x}\alpha_n\right\}j_n(\rho x)\\
       &\;+\sum_{n=0}^{\infty}i^n\left\{\frac{2\alpha_n}{x}-2\alpha_n'+\alpha_nq(x)\right\}\rho j_{n+1}(\rho x)\\
       &\;+[\alpha_0''-q(x)\alpha_0']j_0(\rho x)+q(x)j_0(\rho x)-q(x)\rho xj_1(\rho x)+q(x)i\rho xj_0(\rho x).
   \end{align*}
   Set $\bL=\frac{1}{k(x)}\frac{d}{dx}\kappa(x)\frac{d}{dx}=\frac{d^2}{dx^2}-q(x)\frac{d}{dx}$. From \eqref{eq:2ndformalpower} and \eqref{eq:alpha0} we get $\bL\alpha_0=\bL x-\bL\varphi_{\kappa}^{(1)}=-q(x)$, and hence
   \begin{align*}
       0=& \sum_{n=1}^{\infty}i^n\left\{\alpha_n''+\frac{2n}{x}\alpha_n'+\frac{n(n-1)}{x^2}\alpha_n-q(x)\alpha_n'-q(x)\frac{n}{x}\alpha_n\right\}j_n(\rho x)\\
       &\;+\sum_{n=0}^{\infty}i^n\left\{\frac{2\alpha_n}{x}-2\alpha_n'+\alpha_nq(x)\right\}\rho j_{n+1}(\rho x)\\
       &\;-q(x)\rho xj_1(\rho x)+q(x)i\rho xj_0(\rho x).
   \end{align*}
   A direct computation shows that
   \[
   \frac{1}{x^n}\bL[x^n\alpha_n]=\alpha_n''-q(x)\alpha_n+\frac{2n}{x}\alpha_n'+\left(\frac{n(n-1)}{x^2}-\frac{n}{x}q(x)\right)\alpha_n
   \]
   and using the relation \cite[Formula 10.1.19]{abramovitz}
   \[
   j_n(\rho x)=\frac{\rho x}{2n+1}(j_{n-1}(\rho x)+j_{n+1}(\rho x) ),\quad n\in \mathbb{N},
   \]
   we get
   \begin{align*}
       0=
       &\sum_{n=1}^{\infty}i^n\frac{1}{x^n}\bL[x^n\alpha_n]\frac{\rho x}{2n+1}j_{n-1}(\rho x)+\sum_{n=1}^{\infty}i^n\frac{1}{x^n}\bL[x^n\alpha_n]\frac{\rho x}{2n+1}j_{n+1}(\rho x) \\ 
       & +\sum_{n=0}^{\infty}i^n\left\{\frac{2\alpha_n}{x}-2\alpha_n'+\alpha_nq(x)\right\}\rho j_{n+1}(\rho x)-q(x)\rho xj_1(\rho x)+q(x)i\rho xj_0(\rho x)
   \end{align*}
The first term in the first series is $i\rho L[x\alpha_1]\frac{j_0(\rho x)}{3}$. By \eqref{eq:nthformalpower} and \eqref{eq:alpha1}, we get
\[
\bL[x\alpha_1]=\frac{3}{2}\bL[x^2]-\frac{3}{2}\bL\varphi_{\kappa}^{(2)}=\frac{3}{2}(2-2xq(x)-2)=-3xq(x),
\]
hence $\frac{\bL[x\alpha_1]}{3x}+q(x)=0$ , and we obtain
\begin{align*}
       0=
       &\sum_{n=2}^{\infty}i^n\frac{1}{x^n}\bL[x^n\alpha_n]\frac{\rho x}{2n+1}j_{n-1}(\rho x)+\sum_{n=1}^{\infty}i^n\frac{1}{x^n}\bL[x^n\alpha_n]\frac{\rho x}{2n+1}j_{n+1}(\rho x) \\ 
       & +\sum_{n=0}^{\infty}i^n\left\{\frac{2\alpha_n}{x}-2\alpha_n'+\alpha_nq(x)\right\}\rho j_{n+1}(\rho x)-q(x)\rho xj_1(\rho x).
   \end{align*}
Since $\bL[\alpha_1]=-q$, then
   \begin{align*}
       0=
       &\sum_{n=2}^{\infty}i^n\frac{1}{x^n}\bL[x^n\alpha_n]\frac{\rho x}{2n+1}j_{n-1}(\rho x)+\sum_{n=0}^{\infty}i^n\frac{1}{x^n}\bL[x^n\alpha_n]\frac{\rho x}{2n+1}j_{n+1}(\rho x) \\ 
       & +\sum_{n=0}^{\infty}i^n\left\{\frac{2\alpha_n}{x}-2\alpha_n'+\alpha_nq(x)\right\}\rho j_{n+1}(\rho x).
   \end{align*}
  Reordering indices we conclude that
  \begin{align*}
       0=
       &\sum_{m=1}^{\infty}i^{m+1}\frac{1}{x^{m+1}}\bL[x^{m+1}\alpha_{m+1}]\frac{\rho x}{2m+3}j_{m}(\rho x) \\ 
       & +\sum_{m=1}^{\infty}i^{m-1}\left\{\frac{1}{x^{m-1}}\bL[x^{m-1}\alpha_{m-1}]+\frac{1}{x}\left(\frac{2\alpha_{m-1}}{x}-2\alpha_{m-1}'+\alpha_{m-1}q(x)\right)\right\}\rho x j_{m}(\rho x).
   \end{align*}
   Dividing between $i\rho x$, we arrive at
   \begin{equation}\label{eq:seriesneumannaux1}
       \sum_{m=1}^{\infty}\zeta_m(x)j_m(\rho x)=0,
   \end{equation}
   where
   \begin{align}
       \zeta_m(x) =& \frac{\bL[x^{m+1}\alpha_{m+1}]}{(2m+3)x^{m-1}}-\frac{\bL[x^{m-1}\alpha_{m-1}]}{(2m-1)x^{m-1}}-\frac{1}{x}\left(\frac{2\alpha_{m-1}}{x}-2\alpha_{m-1}'+\alpha_{m-1}q(x)\right),\quad n\in \mathbb{N}.
   \end{align}
   Since the series \eqref{eq:seriesneumannaux1} vanishes for all $\rho$, in particular for $\rho$ real, we employ the orthogonality property \cite[Formula 11.4.6]{abramovitz}
   \[
     \int_0^{\infty} j_n(z)j_m(z)dz=0,\quad \text{if  } n\neq m, 
   \]
   and we conclude that $\zeta_m(x)=0$ for all $m\in \mathbb{N}$. Reordering indices, this condition is equivalent to
   \begin{equation}\label{eq:1strecursiverel}
      \frac{\bL[x^{m}\alpha_{m}]}{(2m+1)x^{m}}= \frac{\bL[x^{m-2}\alpha_{m-2}]}{(2m-3)x^{m-2}}-\frac{1}{x}\left(\frac{2\alpha_{m-2}}{x}-2\alpha_{m-2}'+\alpha_{m-2}q(x)\right),\quad m\geq 2.
   \end{equation}
   Denote $\sigma_m(x):=x^{m}\alpha_m(x)$, $m\in \mathbb{N}_0$. A direct computation shows that, for $m\geq 2$, $\sigma_m(0)=\sigma_m'(0)=0$. Hence, we rewrite \eqref{eq:1strecursiverel} as
   as
   \[
   \frac{\bL[\sigma_m]}{(2m+1)x^m}=f_m(x),
   \]
   where 
   \begin{align*}
       f_m(x) =& \frac{\bL[\sigma_{m-2}]}{(2m-3)x^{m-2}}+\frac{1}{x}\left(\frac{2\sigma_{m-2}}{x^{m-1}}-2\alpha_{m-2}'+q(x)\frac{\sigma_{m-2}}{x^{m-2}}\right) \\
       = & \frac{\bL[\sigma_{m-2}]}{(2m-3)x^{m-2}}+\frac{1}{x}\left(\frac{2\sigma_{m-2}}{x^{m-1}}-2\left(\frac{\sigma_{m-2}'}{x^{m-2}}-\frac{(m-2)\sigma_{m-2}}{x^{m-1}}\right)+q(x)\frac{\sigma_{m-2}}{x^{m-2}}\right)\\
       =& \frac{\bL[\sigma_{m-2}]}{(2m-3)x^{m-2}}+\frac{1}{x^m}\left[2(m-1)\sigma_{m-2}-2x\sigma_{m-2}'+xq(x)\sigma_{m-2}\right].
   \end{align*}
  Consequently,
  \[
  \bL[\sigma_m]=\frac{2m+1}{2m-3}\left\{x^2L[\sigma_{m-2}]+2(m-1)(2m-3)\sigma_{m-2}-2(2m-3)\sigma_{m-2}'+(2m-3)xq(x)\sigma_{m-2}\right\}.
  \]
  A straightforward calculation shows the equality
  \begin{equation}\label{eq:2ndrecursiverel}
      \bL[\sigma_m]=\frac{2m+1}{2m-3}x^{2m-3}\bL\left[\frac{\sigma_{m-2}}{x^{2m-3}}\right],\quad m\geq 2.
  \end{equation}\\
Since $\sigma_m$ satisfies the initial conditions $\sigma_m(0)=\sigma_{m-1}'(0)=0$, it follows that
\[
\sigma_m(x)=\frac{2m+1}{2m-3}\int_0^x\frac{1}{\kappa(t)}\left(\int_0^t\kappa(s)\left(s^{2m-1}\bL\left[\frac{\sigma_{m-2}(s)}{s^{2m-3}}\right]\right)ds\right)dt.
\]
Denote $D=\frac{d}{dx}$. Hence
\begin{align*}
    \sigma_m(x)= &\frac{2m+1}{2m-3} \int_0^x\frac{1}{\kappa(t)}\left(\int_0^t\kappa(s)s^{2m-1}\left(\frac{1}{\kappa(s)}D\left(\kappa(s)D\left(\frac{\sigma_{m-2}(s)}{s^{2m-3}}\right)\right)\right)ds\right)dt\\
    &=\frac{2m+1}{2m-3} \int_0^x\frac{1}{\kappa(t)}\left(\int_0^ts^{2m-1}D\left(\kappa(s)D\left(\frac{\sigma_{m-2}(s)}{s^{2m-3}}\right)\right)ds\right)dt
\end{align*}
Set $I=\int_0^ts^{2m-1}D\left(\kappa(s)D\left(\frac{\sigma_{m-2}(s)}{s^{2m-3}}\right)\right)ds$. Integration by parts yields
\begin{align*}
    I= & s^{2m-1}\kappa(s)D\left(\frac{\sigma_{m-2}(s)}{s^{2m-3}}\right)\Bigg{|}_0^t-\int_0^t(2m-1)s^{2m-2}\kappa(s)D\left(\frac{\sigma_{m-2}(s)}{s^{2m-3}}\right)ds \\
    =& s^{2m-1}\kappa(s)\left[\frac{\sigma_{m-2}'(s)}{s^{2m-3}}-(2m-3)\frac{\sigma_{m-2}(s)}{s^{2m-2}}\right]\Bigg{|}_0^t-\int_0^t(2m-1)s^{2m-2}\kappa(s)D\left(\frac{\sigma_{m-2}(s)}{s^{2m-3}}\right)ds\\
    =& \kappa(t)\left[t^2\sigma_{m-2}'(t)-(2m-3)t\sigma_{m-2}(t)\right]-(2m-1)s^{2m-2}\kappa(s)\frac{\sigma_{m-2}(s)}{s^{2m-3}}\bigg{|}_0^t\\
    &+(2m-1)\int_0^t\left[(2m-2)s^{2m-3}\kappa(s)+s^{2m-2}\kappa'(s)\right]\frac{\sigma_{m-2}(s)}{s^{2m-3}}ds\\
    =& \kappa(t)\left[t^2\sigma_{m-2}(t)-4(m-1)t\sigma_{m-2}(t)+(2m-1)\theta_m(t)\right],
    \end{align*}
    where $\theta_m(t)=\frac{1}{\kappa(t)}\int_0^t\left[(2m-2)s^{2m-3}\kappa(s)+s^{2m-2}\kappa'(s)\right]\frac{\sigma_{m-2}(s)}{s^{2m-3}}ds$.
    Thus, $J=\int_0^x\frac{1}{\kappa(t)}I(t)dt$ becomes
    \begin{align*}
        J=\int_0^x\left\{t^2\sigma_{m-2}(t)-4(m-1)t\sigma_{m-2}(t)+(2m-1)\theta_m(t)\right\}dt
    \end{align*}
    Since
    \[
    \int_0^xt^2\sigma_{m-2}(t)=t^2\sigma_{m-2}(t)\Big{|}_0^x-\int_0^x2t\sigma_{m-2}(t)dt,
    \]
    we obtain
    \[
    J=x^2\sigma_{m-2}(x)-\int_0^x\left\{2(2m-1)t\sigma_{m-2}(t)-(2m-1)\theta_m(t)\right\}dt.
    \]
    Denoting $\eta_m(x):=\int_0^x\left\{2(2m-1)t\sigma_{m-2}(t)-(2m-1)\theta_m(t)\right\}dt$, we conclude that
    \[
    \sigma_m(x)=\frac{2m+1}{2m-3}\left(x^2\sigma_{m-2}(x)-\eta_m(x)\right).
    \]
Now we see that when $\kappa\in C^2[0,L]$, differentiation of the series twice term by term is justified. By Theorem \ref{Th:transmutationproperty}, $K_{\kappa}\in C^2(\overline{\mathcal{T}_L})$, and hence, the function $G_0(x,z):=K_{\kappa}(x,xz)$, $0\leq x\leq L, -1\leq z\leq 1$ is of class $C^2$ (because it is the composition of $K_{\kappa}$ with the $C^{\infty}$ transformation $[0,L]\times [-1,1]\ni (x,z)\mapsto (x,xz)\in \overline{\mathcal{T}_L}$). Denote $G_1(x,z):=\frac{\partial G_0(x,z)}{\partial x}$ and $G_2(x,z):= \frac{\partial^2G_0(x,z)}{\partial x^2}$. It is clear that these functions are continuous in $[0,L]\times [-1,1]$. Hence $C_j(L):=\max_{0\leq x\leq L}\int_{-1}^1|G_j(x,z)|^2dt<\infty$. Changing variables and using the Leibniz rule for differentiation under the integral sign, we obtain
\[
\alpha_n^{(j)}(x)=(2n+1)\int_{-1}^1G_j(x,z)P_n(z)dz,\quad n\in \mathbb{N}_0, j=0,1,2.
\]
That is, $\{2\alpha_n^{(j)}(x)\}$ are precisely the Fourier Legendre coefficients of the function $G_j(x,\dot)\in L^2(-1,1)$. By Parseval's identity, $\sum_{n=0}^{\infty}\frac{8|\alpha_n^{(j)}(x)|^2}{2n+1}=\int_{-1}^1|G_j(x,\tau)|^2d\tau\leq C_j(L)$. Consequently, Lemma \ref{Lemma:convergencebessel}(ii) implies that the series $\sum_{n=0}^{\infty}\alpha_n^{(j)}(x)j_n(\rho x)$ converges absolutely and uniformly for $0<x\leq L$, $j=0,1,2$. Thus, the series \eqref{eq:NSBFe} may be differentiated twice term by term, as well as re-indexed, ensuring the validity of formulas \eqref{eq:simgamgeq2}. 

Finally, if $\kappa\in W^{1,2}(0,L)$, choose a sequence $\{q_m\}\subset C^2[0,L]$ such that $q_m\rightarrow q$ in $L^2(0,L)$. Taking $\kappa_m(x)=e^{-\int_0^xq_m(s)ds}$, we have that $\kappa_m,\frac{1}{\kappa_m}\rightarrow \kappa,\frac{1}{\kappa}$ uniformly on $[0,L]$ as $m\rightarrow \infty$. Let $\{\varphi_{\kappa_m}^{(k)}\}_{k=0}^{\infty}$ be the corresponding formal powers associated to $\kappa_m$. According to \cite[Th. 10]{mineimpedance3}, for each $k\in \mathbb{N}_0$, $\varphi_{\kappa_m}^{(k)}\rightarrow \varphi_{\kappa}^{(k)}$ uniformly on $[0,L]$ as $m\rightarrow \infty$. From formula \eqref{eq:alphaintermsofphi}, it is clear that the corresponding coefficients satisfy $\sigma_{m,n}\rightarrow \sigma_n$, $m\rightarrow \infty$, uniformly on $[0,L]$. Passing to the limit as $m\rightarrow \infty$ in \eqref{eq:sigma0}-\eqref{eq:simgamgeq2}, we conclude that $\sigma_m(x)$ satisfies the recursive relations.
\end{proof}

\section{Solution of boundary valued problems}
From this section onward, we assume that $\kappa(x) > 0$ for all $0<x<L$.  
We recall that $L^2_{\kappa}(0,L)$ denotes the Hilbert space of square-integrable functions with respect to the weight function $\kappa(x)$. The associated inner product is defined by

\begin{equation}
    \langle f, g \rangle_{L^2_{\kappa}} = \int_0^L f(x) \overline{g(x)}  \kappa(x) \, dx.
\end{equation}

\subsection{The Dirichlet problem}

We denote by $\LD: \DomLD \subset L^2_{\kappa}(0,L) \to L^2_{\kappa}(0,L)$ 
the operator associated with the Dirichlet problem, whose domain is
\begin{equation}
    \DomLD= \{ u \in W^{2,2}(0,L) \mid u(0) = u(L) = 0 \}
\end{equation}
and whose action is given by
\begin{equation}
     \LD  u = -\frac{1}{\kappa(x)} \frac{d}{dx} \left( \kappa(x) \frac{du(x)}{dx} \right).
\end{equation}

\begin{remark}
    Using integration by parts and the boundary conditions, it follows that $\LD$ is Hermitian
    \begin{equation}
         \langle \LD u,v \rangle_{L^2_{\kappa}} = \langle u, \LD v \rangle_{L^2_{\kappa}}, \quad \forall u, v \in \DomLD.
    \end{equation}
\end{remark}

\begin{proposition}
    The operator $\LD$ is self-adjoint in $L_\kappa^2(0,L)$
\end{proposition}
\begin{proof}
    By \cite[Lemma 18]{mineimpedance3}, the operator $\LD$ is densely defined in $L_{\kappa}^2(0,L)$. We denote its adjoint by $\mathbf{L}^*$. Let $u\in \mathscr{D}(\mathbf{L}^*)$ and set $v=\mathbf{L}^*u$. By definition of the adjoint operator, we have
    \begin{equation}\label{eq:auxiliarselfadjoint}
       \langle \LD \phi, u\rangle_{L_{\kappa}^2}=\langle \phi, v\rangle_{L_{\kappa}^2}\qquad \forall \phi \in \DomLD.
    \end{equation}
    Consider the operator $\mathbf{R}^D: L_{\kappa}^2(0,L)\rightarrow L_{\kappa}^2(0,L)$ defined by
    \[
       \mathbf{R}^Dg(x):= -\int_0^x\frac{dt}{\kappa(t)}\int_0^tg(s)\kappa(s)ds+\frac{\int_0^L\frac{dt}{\kappa(t)}\int_0^tg(s)\kappa(s)ds}{\int_0^L\frac{dt}{\kappa(t)}}\int_0^x\frac{dt}{\kappa(t)} \qquad \text{for } g\in L_{\kappa}^2(0,L)
    \]
    (the denominator is well defined because $\kappa>0$). By construction, $\mathbf{R}^D(L_{\kappa}^2(0,L))\subset \mathscr{D}(\LD)$ and $\LD \mathbf{R}^D g=g$ for all $g\in L_{\kappa}^2(0,L)$. Given $g\in L_{\kappa}^2(0,L)$ arbitrary, choosing $\phi=\mathbf{R}^Dg$ in \eqref{eq:auxiliarselfadjoint}, and using that $\LD$ is Hermitian, we get
    \begin{align*}
        \langle g, u\rangle_{L_{\kappa}^2}&=\langle \LD \mathbf{R}^Dg, u\rangle_{L_{\kappa}^2}=\langle \mathbf{R}^Dg, v\rangle_{L_{\kappa}^2}=\langle \mathbf{R}^Dg, \LD\mathbf{R}^Dv\rangle_{L_{\kappa}^2}=\langle \LD \mathbf{R}^Dg, \mathbf{R}^Dv\rangle_{L_{\kappa}^2}=\langle g, \mathbf{R}^Dv\rangle_{L_{\kappa}^2}.
    \end{align*}
    Since $g$ is arbitrary, it follows that $u=\mathbf{R}^Dv\in \DomLD$ and $\LD u=v$. Therefore, $\LD=\mathbf{L}^*$.
\end{proof}
\begin{theorem}\label{th:eigenvalores}
	The eigenvalues of the Dirichlet problem satisfy the following statements:
    \begin{enumerate}
        \item The eigenvalues are real and positive.
        \item Let $\lambda=\rho^2$. Then $\lambda$ is an eigenvalue iff satisfy the characteristic equation:
        \begin{equation}
             S_{\kappa}(\rho, L) = 0.
        \end{equation}
        \item  There exists a countable set of eigenvalues  $\{ \lambda_n \}_{n=1}^{\infty}$ that form an increasing sequence tending to infinity:
            \begin{equation}\label{eq:increasinglambda}
               0 < \lambda_1 < \lambda_2 < \cdots < \lambda_n \to \infty, \quad n \to \infty.
            \end{equation}
            Furthermore, the spectral values $\rho_n=\sqrt{\lambda_n}$ satisfies the asymptotic conditions
            \begin{equation}\label{eq:asymptoticrho}
                \rho_n= \frac{n\pi}{L}+\zeta_n \;\;\;\; \text{for } n \; \text{large enough, where  } \{\zeta_n\}\in \ell_2.
            \end{equation}
        \item Each $\lambda_n=\rho_n^2$ has a unique normalized associated eigenfunction $\phi_n(x)$ given by:
        \begin{equation}
            \phi_n(x) = \frac{S_{\kappa}(\rho_n, x)}{\sqrt{\beta_n}}, \quad n \in \mathbb{N},
        \end{equation}
        where $\{\beta_n\}_{n=1}^\infty$ are the normalization constants given by
        \begin{equation}
            \beta_n = \int_0^L S_{\kappa}^2(\rho_n, x) \kappa(x) dx,\quad n \in \mathbb{N}.
        \end{equation}
         \item The eigenfunctions $\{\phi_n\}_{n=1}^\infty$ forms an orthonormal basis for $L^2_{\kappa}(0,L)$. Moreover, if $f \in \DomLD$, then the series 
        \begin{equation}
            f= \sum_{n=1}^{\infty} \langle f, \phi_n \rangle_{L^2_{\kappa}} \phi_n
        \end{equation}
        converges uniformly on $[0,L]$.
        
    \end{enumerate}
\end{theorem}

\begin{proof}
\begin{enumerate}
    \item Since $\LD$ is self-adjoint, the eigenvalues are real. On the other hand, given an eigenvalue $\lambda$ and a corresponding eigenfunction $u$, considering $\langle \LD u, u \rangle_{L^2_{\kappa}}$, integration by parts and the boundary conditions imply that

    \begin{equation*}
        \langle \LD u, u \rangle_{L^2_{\kappa}} = -\int_0^L  (\kappa(x) u(x)')' \overline{u(x)} dx = \int_0^L \kappa(x) |u'(x)|^2 dx = \| u' \|_{L^2_{\kappa}}^2.
    \end{equation*}
    
    Due to the Dirichlet boundary conditions, no eigenfunction can be constant, which implies that $u' \neq 0$ and therefore $\| u' \|_{L^2_{\kappa}} > 0$. Thus, $\langle \LD u, u \rangle_{L^2_{\kappa}} > 0.$ On the other hand,
    
    \begin{equation*}
        \langle \LD u, u \rangle_{L^2_{\kappa}} = \langle \lambda u, u \rangle_{L^2_{\kappa}} = \lambda \| u \|_{L^2}^2.
    \end{equation*}
    
    Therefore,
    \begin{equation*}
        \| u' \|_{L^2_{\kappa}}^2 = \lambda \| u \|_{L^2}^2.
    \end{equation*}
    
    Since $u \neq 0$, we conclude that
    \begin{equation*}
        \lambda = \frac{\| u' \|_{L^2}^2}{\| u \|_{L^2}^2} > 0.
    \end{equation*}

    \item \label{item:2}      
    Let $\lambda = \rho^2$. Every solution $u$ of the Eq. \eqref{eq:SLEIF} has the form:
    
    \begin{equation}
        u = a C_{\kappa}(\rho, x) + b S_{\kappa}(\rho, x),
    \end{equation}
    
    where $a$ and $b$ are constants. From the conditions \eqref{eq:conditionscosine} and \eqref{eq:conditionssine}, $u$ satisfies the Dirichlet conditions iff $u = b S_{\kappa}(\rho, x)$ and it is an eigenfunction iff $b\neq 0$ and $S_{\kappa}(\rho, L) = 0.$

    \item   Since $\kappa, 1/\kappa \in L^2(0,L)$, the problem is regular and has compact resolvent, and \eqref{eq:increasinglambda} follows from \cite[Th. 2.74]{miklavic}. For \eqref{eq:asymptoticrho}, using integration by parts and the Goursat conditions \eqref{eq:goursatcond} in the integral representation \eqref{eq:integralrep} for $e_{\kappa}(\rho,L)$, we get
    \[
    e_{\kappa}(\rho,L)=\frac{e^{i\rho L}}{\sqrt{\kappa(L)}}+\int_{-L}^{L}\frac{\partial K_{\kappa}(L,t)}{\partial t}e^{i\rho t}.
    \]
    From \eqref{eq:definitionsolutioncosineandsine}, we deduce the following integral representation for the characteristic equation:
    \[
    S_{\kappa}(\rho,L)=\frac{\sin(\rho L)}{\sqrt{\kappa(L)}\rho}+\frac{1}{\rho}\int_{-L}^{L}k_L^{-}(t)e^{i\rho t}dt
    \]
where $k_L^+(t):=\frac{1}{2i}\left(\frac{\partial K_{\kappa}(L,t)}{\partial t}+\frac{\partial K_{\kappa}(L,-t)}{\partial t}\right)$. 
Since $\{\rho_n\}_{n=1}^{\infty}$ are the zeros of the function $\sqrt{\kappa(L)}\rho S_{\kappa}(\rho,L)$ and the kernel $\sqrt{\kappa(L)}k_L^+\in L^2(-L,L)$, it follows from \cite{levin}  that $\{\rho_n\}_{n=1}^{\infty}$ satisfies the asymptotic \eqref{eq:asymptoticrho}.
    \item  From the point~\ref{item:2}, if $\lambda_n = \rho_n^2$ is an eigenvalue, then every eigenfunction associated with $\lambda_n$ is a multiple of $S_{\kappa}(\rho_n,x)$. Hence, the unique eigenfunction normalized in $L^2_\kappa(0,L)$ is given by
    \begin{equation}
        \phi_n(x) = \frac{S_{\kappa}(\rho_n,x)}{\sqrt{\beta_n}}, \quad \text{where } \beta_n = \|S_{\kappa}(\rho_n,\cdot)\|_{L^2_\kappa}^2.
    \end{equation}

    \item  Since the problem is regular, the system $\{\phi_n\}_{n=1}^\infty$ forms a complete orthonormal basis in $L^2_\kappa(0,L)$; this follows from~\cite[Th.2.74]{miklavic}. Moreover, the uniform convergence of the series expansion for $f \in \DomLD$ follows from~\cite[Cor.4.3.3]{benewitz}.
\end{enumerate}
\end{proof}

\begin{remark}
    The set $\{\lambda_n,\beta_n\}_{n=1}^{\infty}$ is called the {\bf direct spectral data} of the Dirichlet problem.
\end{remark}
The characteristic equation for the Dirichlet Problem is given by $S_\kappa(\rho,L)=0$. Hence, by formula \eqref{eq:NSBFseno}, this equation admits an NSBF representation and the eigenvalues $\lambda = \rho^2$ are obtained by solving the equation:

\begin{equation}
    0 = \frac{\sin(\rho L)}{\rho} - \sum_{n=0}^{\infty} (-1)^n \alpha_{2n}(L) j_{2n}(\rho L).
\end{equation}

\subsection{The Weyl function}
Let $\lambda \in \mathbb{C}$ be such that is not an eigenvalue of the Dirichlet Problem. In this case, there exists a unique solution $\Phi_{\kappa}(\lambda,x)$ of the following boundary value problem

\[
\begin{cases}
\displaystyle -\frac{1}{\kappa(x)} \frac{d}{dx}\kappa(x)\frac{d}{dx}\Phi_{\kappa}(\lambda,x) = \lambda \Phi_{\kappa}(\lambda,x), \\
    \Phi_{\kappa}(\lambda,0)=1, \quad \Phi_{\kappa}(\lambda,L)=0.
\end{cases}
\]  
    A direct computation shows that 
    \begin{equation}
        \Phi_{\kappa}(\lambda,x)=C_\kappa(\rho,x)-\frac{C_\kappa (\rho,L)}{S_\kappa (\rho,L)} S_\kappa(\rho,x),
        \label{eq:Weylsolution}
    \end{equation}
    where $\rho^2=\lambda$. This formula is well defined because $S_\kappa(\rho,x)\neq 0.$ The coefficient appearing in \eqref{eq:Weylsolution} has a distinguished interpretation in Sturm–Liouville theory.
\begin{definition}
   The Weyl function associated with the Dirichlet-problem is given by
   \begin{equation} \label{def:Weylfunction}
        M_\kappa^D(\lambda):= \frac{C_\kappa(\rho,L)}{S_\kappa(\rho,L)}
   \end{equation}
   for $\lambda$ that are not eigenvalues of the Dirichlet Problem. 
\end{definition}
The following proposition is standard, but we include their proof for the sake of completeness.
\begin{proposition}
    The Weyl function $M_{\kappa}^D$ is a meromorphic function of $\lambda$ with simple poles at the eigenvalues $\{\lambda_n\}_{n=1}^\infty$ of the Dirichlet problem and the residue at each eigenvalue is given by
    \begin{equation}\label{eq:residueWeyl}
         \operatorname{Res}_{\lambda=\lambda_n}M_\kappa^D(\lambda)=\frac{1}{\beta_n} \qquad \forall n\in \mathbb{N}.
    \end{equation}
\end{proposition}
\begin{proof}
Let $\lambda\in \mathbb{C}$. By the SPPS method \cite{spps}, the functions $C_{\kappa}(\rho,x)$ and $S_{\kappa}(\rho,x)$ are entire in $\lambda$. Hence, $M_{\kappa}^D(\lambda)$ is meromorphic with poles located at the zeros of $\Delta_D(\lambda):=S_{\kappa}(\rho,L)$, which coincide with the eigenvalues $\{\lambda_n\}_{n=1}^{\infty}$ of the Dirichlet problem.
    Given $\lambda\in \mathbb{C}$, consider the solution $\vartheta_{\kappa}(\lambda,x)$ of \eqref{eq:SLEIF} satisfying the  conditions
    \[
    \vartheta_{\kappa}(\lambda,L)=0,\quad \vartheta_{\kappa}(\lambda,L)=\frac{1}{\kappa(L)}.
    \]
 A direct computation shows that
 \[
 \vartheta_{\kappa}(\lambda,x)=C_{\kappa}(\rho,L)S_{\kappa}(\rho,x)-S_{\kappa}(\rho,L)C_{\kappa}(\rho,x).
 \]
 Note that $\Delta_D(\lambda)=-\vartheta(\lambda,0)$, and if $\lambda=\lambda_n$, then
 \begin{equation}\label{eq:coefficientvartheta}
     \vartheta(\lambda_n,x)=C_{\kappa}(\rho_n,L)S_{\kappa}(\rho_n,x),
 \end{equation}
 which implies that $C_{\kappa}(\rho_n,L)\neq 0$. Integration by parts yields
 \begin{align*}
    (\lambda-\lambda_n) \langle \vartheta_{\kappa}(\lambda,\cdot), S_{\kappa}(\rho_n,x) \rangle_{L_{\kappa}^2}=& \langle \bL\vartheta_{\kappa}(\lambda,\cdot), S_{\kappa}(\rho_n,x) \rangle_{L_{\kappa}^2}-\langle \vartheta_{\kappa}(\lambda,\cdot), \bL S_{\kappa}(\rho_n,x) \rangle_{L_{\kappa}^2}\\
    =&\left[\vartheta_{\kappa}(\lambda,x)\kappa(x)S_{\kappa}'(\rho_n,x)-\kappa(x)\vartheta'_{\kappa}(\lambda,x)S_{\kappa}(\rho_n,x)\right]_0^L
    =-\vartheta(\lambda,0)
 \end{align*}
 Thus, $\frac{\Delta(\lambda)}{\lambda-\lambda_n}=\langle \vartheta_{\kappa}(\lambda,\cdot), S_{\kappa}(\rho_n,x) \rangle_{L_{\kappa}^2}$. Passing to the limit as $\lambda\rightarrow \lambda_n$ and using \eqref{eq:coefficientvartheta}, we obtain
 \[
 \frac{d\Delta_D(\lambda)}{d\lambda}\bigg{|}_{\lambda=\lambda_n}=C_{\kappa}(\rho_n,L)\beta_n \quad \forall n\in \mathbb{N}.
 \]
 Then each $\lambda_n$ is a simple zero of $\Delta_D(\lambda)$, and consequently, a simple pole of $M_D(\lambda)$, which implies that formula \eqref{eq:residueWeyl} holds.
\end{proof}\\
By the NSBF representations \eqref{eq:NSBFcoseno} and \eqref{eq:NSBFseno} of the solutions $C_\kappa(\rho,x)$ and $S_\kappa(\rho,x)$, we obtain the following representation for the Weyl function
\begin{equation}
    M_\kappa^D(\lambda)=\frac{\cos(\rho L) \;+\; \rho \displaystyle \sum_{k=0}^{\infty} (-1)^k \, \alpha_{2k+1}(L) \, j_{2k+1}(\rho L)}{\frac{\sin(\rho L)}{\rho} \;+\; \displaystyle \sum_{k=0}^{\infty} (-1)^{k+1} \, \alpha_{2k}(L) \, j_{2k}(\rho L)}.
    \label{eq:NSBFweyl}
\end{equation}
Of course, an approximation for the numerical computation is given by 
\begin{equation}
    M_{\kappa,N}^D(\lambda)=\frac{\cos(\rho L) \;+\; \rho \displaystyle \sum_{k=0}^{\lfloor (N-1)/2 \rfloor} (-1)^k \, \alpha_{2k+1}(L) \, j_{2k+1}(\rho L)}{\frac{\sin(\rho L)}{\rho} \;+\; \displaystyle \sum_{k=0}^{\lfloor N/2 \rfloor} (-1)^{k+1} \, \alpha_{2k}(L) \, j_{2k}(\rho L)}.
    \label{eq:AproximationNSBFweyl}
\end{equation}

\subsection{Relation with the Neumann problem}
Now we consider the problem with Neumann conditions
\[
u'(0) = u'(L) = 0.
\]
A direct computation shows that the eigenvalues $\mu = \rho^2$ are given by the zeros of the characteristic equation
\[
C_k'(\rho,L) = 0.
\]
Proceeding as in the proof of Theorem~\ref{th:eigenvalores}, we obtain an increasing sequence of eigenvalues $\{\mu_n = \rho_n^2\}_{n=0}^{\infty}$ with the peculiarity that $\mu_0 = 0$ with eigenfunction $\psi_0 \equiv 1$.

The corresponding eigenfunctions and normalizing constants for $\mu_n > 0$ are given by
\begin{equation*}
    \psi_n(x) = \frac{C_k(\rho_n x)}{\sqrt{\gamma_n}}, \quad \gamma_n = \int_0^L C_k^2(\rho_n x)\, \kappa(x)\, dx.
\end{equation*}
The Neumann problem can be reduced to Dirichlet problem with impedance equation $1/a$, via a Darboux transformation.

\begin{remark}[Darboux transform associated to the impedance equation]
    Following \cite{mineimpedance3}, we consider the {\bf Darboux transform} defined as follows. Let $u \in W^{2,2} (0,L)$ be a solution of Eq. \eqref{eq:SLEIF}.
    The function 
    \begin{equation}\label{eq:Daboux}
        v(x)=\kappa(x)u'(x)
    \end{equation}
is called the Darboux transform of $u$. According to \cite[Remark 2]{mineimpedance3}, $v \in W^{2,2}(0,L)$ and satisfies the {\bf Darboux associated equation}
    \begin{equation}\label{eq:DarbouxAssociated}
        -\kappa(x) \frac{d}{dx} \left( \frac{1}{\kappa(x)} \frac{dv}{dx}\right) =\lambda v.
    \end{equation}
In the particular case where $u$ is $C_\kappa(\rho,x)$, its Darboux transform is given by 
\begin{equation*}
    \kappa(x)C'_\kappa(\rho,x)=-\rho^2 S_{1/\kappa (\rho,x)}
\end{equation*}
(see \cite[Remark 9]{mineimpedance3}), from where we obtain that
 \begin{equation}\label{eq:DarbouxCk}
     C'_\kappa(\rho,x)= -\frac{\rho^2 S_{1/\kappa  }(\rho,x)}{\kappa(x)}.
 \end{equation}
Similarly, \cite[Remark 9]{mineimpedance3} establishes that
 \begin{equation}\label{eq:DarbouxSk}
     S'_\kappa(\rho,x)= \frac{C_{1/\kappa  }(\rho,x)}{\kappa(x)}.
 \end{equation}
\end{remark}

\begin{theorem}
The positive eigenvalues of the Neumann problem with conductivity function $\kappa(x)$ coincide with the eigenvalues of the Dirichlet problem with conductivity $1/\kappa(x)$. Moreover, the corresponding normalizing constants $\{\gamma_n^{\kappa}\}_{n=0}^{\infty}$ satisfy 
\begin{equation}\label{eq:normalizingdirichlet}
        \gamma_n^{\kappa}=\mu_n \beta_n^{1/\kappa} \qquad \forall n\in \mathbb{N}.
\end{equation}
\end{theorem}

\begin{proof}
    If $\mu = \rho^2 > 0$, then it is an eigenvalue of the Neumann problem with conductivity $\kappa$ iff it satisfies the characteristic equation 
\[
C'_{\kappa}(\rho, L) = 0.
\]
According to Eq. \eqref{eq:DarbouxCk}
\begin{equation}\label{eq:auxiliarchareqck}
    C'_{\kappa}(\rho, L) =  \frac{- \mu S_{1/\kappa}(\rho, L)}{\kappa(L)}.
\end{equation}
Since $\kappa(L) > 0$, and $\mu > 0$ by hypothesis, then $C'_{\kappa}(\rho, L) = 0$ iff $S_{1/\kappa}(\rho, L) = 0$, i.e. iff $\mu$ is an eigenvalue of the Dirichlet problem with conductivity $1/\kappa$. For $n\geq 1$, let $\gamma_n^{\kappa}$ be the normalizing constant associated with $\mu_n=\rho_n^2$. Using relation \eqref{eq:DarbouxSk}, integration by parts together with the equality \eqref{eq:auxiliarchareqck} yields
\begin{align*}
  \gamma_n^{\kappa}& = \int_0^L C_{\kappa}^2(\rho_n,x)\kappa(x)dx= \int_0^L \left(\frac{1}{\kappa(x)}S'_{1/\kappa}(\rho_n,x)\right)S'_{1/\kappa}(\rho_n,x)dx \\
  & = \frac{1}{\kappa(x)}S'_{1/\kappa}(\rho_n,x)S_{1/\kappa}(\rho_n,x)\bigg{|}_0^L-\int_0^L \left(\frac{1}{\kappa(x)}S'_{1/\kappa}(\rho_n,x)\right)'S_{1/\kappa}(\rho_n,x)dx\\
  &= \mu_n \int_0^L S_{1/\kappa}^2(\rho_n,x)\frac{dx}{\kappa(x)}= \mu_n \beta_n^{1/\kappa}.
\end{align*}
\end{proof}\\
On this way, the problem of calculating the positive eigenvalues $\{\mu_{n}\}_{n=1}^{\infty}$ of the Neumann problem is reduced to the calculation of the corresponding eigenvalues of the Dirichlet problem with conductivity $1/\kappa$.\\
Concerning the Weyl function for the Neumann problem, this is defined as follows. For $\mu\in \mathbb{C}$ that is not an eigenvalue of the Neumann problem,  let $\Psi_{\kappa}(\mu,x)$ be the unique solution of Eq. \eqref{eq:SLEIF} satisfying the boundary conditions 
\[
\Psi_{\kappa}'(\mu,0)=1, \quad \Psi_{\kappa}'(\mu,L)=0.
\]
A direct computation shows that
\[
  \Psi_{\kappa}(\mu,x)=S_{\kappa}(\mu,x)-\frac{S'_{\kappa}(\mu,L)}{C'_{\kappa}(\mu,L)}C_{\kappa}(\mu,x).
\]
The Weyl function of the Neumann problem is defined as 
\begin{equation}\label{eq:Weylneumann}
    M_{\kappa}^N(\mu):= \frac{S_{\kappa}'(\rho,L)}{C'_{\kappa}(\rho,L)}.
\end{equation}
\begin{proposition}
    Let $\mu\in \mathbb{C}$ be such that is not an eigenvalue of the Neumann problem. Then the following relation holds:
    \begin{equation}\label{eq:Weylneumanntodirichlet}
        M_{\kappa}^N(\mu)=-\frac{1}{\mu}M_{1/\kappa}^D(\mu).
    \end{equation}
\end{proposition}
\begin{proof}
    It follows from the definitions \eqref{eq:Weylneumann} and \eqref{def:Weylfunction} and from formulas \eqref{eq:DarbouxCk} and \eqref{eq:DarbouxSk}.
\end{proof}

\begin{remark}
    Similar relations can be obtained for Sturm-Liouville problems with mixed boundary conditions, namely Neumann–Dirichlet conditions $u'(0)=u(L)=0$ (for which, the characteristic equation is $C_{\kappa}(\rho,L)=0$) and Dirichlet to Neumann conditions $u(0)=u'(L)=0$ (with characteristic equation $S_{\kappa}'(\rho,L)=0$). \\
    In fact, the Darboux relations \eqref{eq:DarbouxCk} and \eqref{eq:DarbouxSk} can be used to obtain NSBF representations for  problems with general separated boundary conditions
    \[
      a_1 u(0)+b_1u'(0)=0,\quad a_2 u(L)+b_2u'(L)=0, \quad \text{where } a_i,b_i\in \mathbb{R} \text{ and  }|a_i|+|b_i|>0, i=1,2.
    \]
\end{remark}

\section{Numerical solution}
\subsection{Implementation details}

In this subsection, we describe the numerical procedure used to approximate the eigenvalues and eigenfunctions of the Dirichlet problem associated with Sturm-Liouville equation in impedance form. This method relies on the NSBF representation and the use of the characteristic equation $S_{\kappa}(\rho, L) = 0$ derived from it.\\
The eigenvalues $\lambda = \rho^2$ are obtained by solving the equation:

\begin{equation}
    0 = \frac{\sin(\rho L)}{\rho} - \sum_{n=0}^{\infty} (-1)^n \alpha_{2n}(L) j_{2n}(\rho L),
\end{equation}
where $\{\alpha_n(x)\}$ are the coefficients of the Furier-Legendre expansion. \\
We approximate the eigenvalues by means of the roots of the truncated series up to an order $N$:

\begin{equation}\label{eq:caracteristicaj0}
    0 =L j_0(\rho L) - \sum_{n=0}^{\lfloor N/2 \rfloor} (-1)^n \alpha_{2n}(L) j_{2n}(\rho L),
\end{equation}
where we used 
$\frac{\sin(\rho x)}{\rho}=xj_0(\rho x)$. 
This truncated characteristic function is evaluated over a discrete set of $\rho$ values to detect sign changes and approximate roots. The purpose of the following pseudocode is to provide a high-level description of the
algorithm implemented, allowing the reader to reconstruct or adapt the method according
to their specific needs.

\begin{algorithm}
\textbf{Pseudocode for the numerical computation of eigenvalues and eigenfunctions:}
\begin{enumerate}
    \item Fix $L > 0$ and define a uniform mesh $x_i \in [0, L]$.
    \item Define the conductivity function $\kappa(x)$ and its derivative $\kappa'(x)$ over the mesh.
    \item Compute the matrix of coefficients NSBF $\Phi_n=(\alpha _n (x_j))$ using the recurrence relations for $\alpha_n:= \frac{\sigma_n(x)}{x^n}$ and compute $(\sigma_n(x))$ using the relations~\ref{eq:sigma0}-\ref{eq:simgamgeq2}.
    \item Let $\Psi := \Phi[:, N]$ be the last column of $\Phi$.
    
    \item Define the characteristic function using the truncated series from equation~\eqref{eq:caracteristicaj0}: \\
    $S_N(\rho)=Lj_0(\rho L)-\sum_{n=0}^{\lfloor N/2 \rfloor}(-1)^n \Psi_{2n} j_{2n}(\rho L)$.
    
    \item Evaluate $S_N(\rho)$ over a suitable interval of $\rho$ values.
    \item Detect sign changes in $S_N(\rho)$ to locate intervals containing roots.
    \item Apply a root-finding method to find $\rho_n$ such that $S_N(\rho_n) = 0$.
    \item Set $\lambda_n = \rho_n^2$ as the corresponding eigenvalue.
    \item For each $\lambda_n$, compute the eigenfunction $S_N(\rho_n,x)$ via the NSBF $N-$partial sum.
\end{enumerate}
\end{algorithm}

\subsection{Numerical examples}

Before presenting the numerical examples, we briefly describe the computational procedure carried out for the implementation of the NSBF method.\\
All calculations were performed using \texttt{Python}. The recursive integrals required for computing the Fourier-Legendre coefficients $\{\alpha_n(x)\}$ were evaluated using Newton–Cotes integration of order 6. A uniform mesh of 2000 points was used in the interval $[0,L]$ to discretize the domain and ensure stability in the numerical integration.\\
Once the coefficients were computed, we proceeded to calculate the zeros of the truncated characteristic function \eqref{eq:caracteristicaj0}. To do this, we discretized the spectral parameter $\rho$ over the interval $[0,200]$ using a uniform mesh of 500 points and 120 coefficients. The interval was subdivided, and sign changes of the characteristic function were detected to locate subintervals containing zeros. Each zero was then refined using \texttt{scipy.optimize.root\_scalar}, a reliable root-finding function from the \texttt{scipy.optimize} module. We used the function \texttt{np.set\_printoptions(precision=14)} to get a better precision.\\
This procedure allowed us to compute the eigenvalues efficiently and with high precision. All computations were carried out on a standard laptop equipped with an Intel(R) Celeron(R) N4000 CPU @ 1.10GHz and 4 GB of RAM. The numerical process was fast and completed in fractions of a second.\\
We now proceed to present some numerical examples, starting with a problem whose exact characteristic equation is known.

\begin{example}[Impedance $a(x)=(1+x)^2$]
The impedance \( a(x) = (1+x)^2 \) exemplifies non-uniform media (e.g., acoustic horns) with analytically tractable yet non-trivial eigenvalues, validating the NSBF method's precision for singular problems.  
In this case, the corresponding impedance equation can be transformed into a Schrödinger equation via the Liouville transformation as follows. If $u$ is a solution of 
\[
- \frac{1}{(1+x)^4} \frac{d}{dx} \left( (1+x)^4 \frac{du}{dx} \right) = \lambda u,
\]
then $v = (1+x)^2 u$ is a solution of the Schrödinger equation
\begin{equation}\label{eq:ESLSchrodinger}
    - v'' + Q_a(x)v = \lambda v, \quad \text{where } \; Q_a(x) = \frac{a''}{a} = \frac{2}{(1+x)^2},
\end{equation}
see \cite[Sec.~3.4]{mineimpedance1}. The Dirichlet conditions are preserved under this transformation. The Dirichlet problem for the SLEIF is transformed into the Dirichlet problem for (\ref{eq:ESLSchrodinger}). In this case, the characteristic equation of the problem is known explicitly and is given by
\begin{equation}
    \Phi(\rho) = \frac{\sin (\rho L)}{\rho} + \frac{1}{\rho}\,\frac{L^2}{(1+L)} j_1(\rho L) = 0,
\end{equation}
see \cite[Subsec. 3.3]{kravchenkodarboux} and \cite{Rodcross}.
In order to compute the exact eigenvalues, we proceed by analyzing the zeros of the characteristic function. The interval is divided into subintervals and we check how the sign of the function changes. Each change of sign indicates the presence of a zero, which corresponds to an eigenvalue. \\
This procedure provides a systematic way to obtain the sequence of exact eigenvalues. In practice, it is implemented numerically by evaluating the characteristic function at a sufficiently dense set of points and applying a root-finding method whenever a sign change is detected. 

\begin{table}[h!]\label{Tab:ejemplo1}
\centering
\begin{tabular}{cccc}
\toprule
$n$ & $\lambda_{\text{exact}}$ & $\lambda_{\text{NSBF}}$ & $|\lambda_{\text{exact}} - \lambda_{\text{NSBF}}|$ \\
\midrule
1  & 1.34805063504836 & 1.34805063504823 & 1.29e-13 \\
2  & 4.43028410885793 & 4.43028410885799 & 6.00e-14 \\
3  & 9.45645939050872 & 9.45645939050879 & 7.00e-14 \\
4  & 16.4672876486339 & 16.4672876486346 & 7.00e-13 \\
5  & 25.4726699599228 & 25.4726699599230 & 2.00e-13 \\
6  & 36.4757027586749 & 36.4757027586759 & 1.00e-12 \\
20 & 400.482239127963 & 400.482239127971 & 8.00e-12 \\
21 & 441.482300968039 & 441.482300968044 & 5.00e-12 \\
22 & 484.482354595155 & 484.482354595153 & 2.00e-12 \\
23 & 529.482401400839 & 529.482401400838 & 1.00e-12 \\
30 & 1024.48264505729 & 1024.48264505730 & 1.00e-11 \\
\bottomrule
\end{tabular}
\caption{Comparison of exact eigenvalues and NSBF-computed eigenvalues for the impedance equation with $a(x) = (1+x)^2$. Absolute errors are reported to quantify the accuracy.}
\label{tab:eigenvalues_comparison}
\end{table}

\begin{figure}[h!]
    \centering
    \includegraphics[width=1\linewidth]{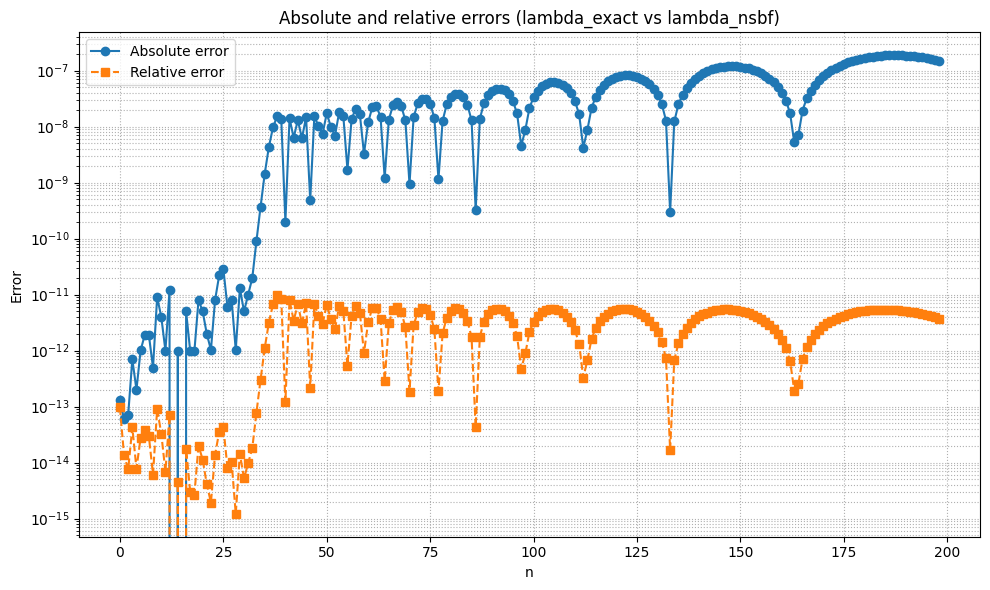}
    \caption{Absolute and relative errors between the exact eigenvalues and those computed with the NSBF method (See Table~\ref{Tab:ejemplo1}). A logarithmic scale is used on the vertical axis to highlight the error magnitude across several orders.}
    \label{fig:Ejemplo1_Errores}
\end{figure}
\FloatBarrier 
Figure~\ref{fig:Ejemplo1_Errores} displays the absolute and relative errors obtained by comparing the eigenvalues computed with the NSBF method against those obtained from the characteristic equation. The logarithmic scale on the vertical axis highlights the error behavior across several orders of magnitude. For the lowest modes, the absolute errors lie in the range $10^{-14}$--$10^{-13}$, essentially reaching the limit of double-precision floating-point arithmetic. Indeed, in IEEE-754 binary64 format only about 15--16 significant decimal digits can be reliably represented, corresponding to a machine epsilon of approximately $2.22 \times 10^{-16}$~\cite{Goldberg1991}. This indicates that the discrepancies observed at the fundamental modes are not attributable to the NSBF method itself, but rather to the intrinsic limits of numerical representation in finite precision arithmetic. As the index $n$ increases, the absolute error grows progressively and stabilizes around $10^{-8}$, which is consistent with the accumulation of numerical round-off and with the oscillatory structure of the NSBF expansion. In contrast, the relative error remains nearly uniform in the range $10^{-12}$--$10^{-11}$ even for higher modes.\\
From a physical and mathematical perspective, the eigenvalues computed in this example can be interpreted as characteristic spectral parameters associated with the underlying SL operator, while the corresponding eigenfunctions represent its spatial modes. In problems of wave propagation in inhomogeneous media, equations of this type naturally arise, with the impedance coefficient accounting for spatial variations of the medium. In such settings, the eigenfunctions describe standing-wave patterns adapted to the local properties of the medium, and the eigenvalues determine the admissible resonant frequencies. Accordingly, the lowest eigenvalues correspond to the fundamental modes of the system, whereas higher modes exhibit an increasing number of nodes.
\begin{figure}[h!]
    \centering
    \includegraphics[width=.77\linewidth]{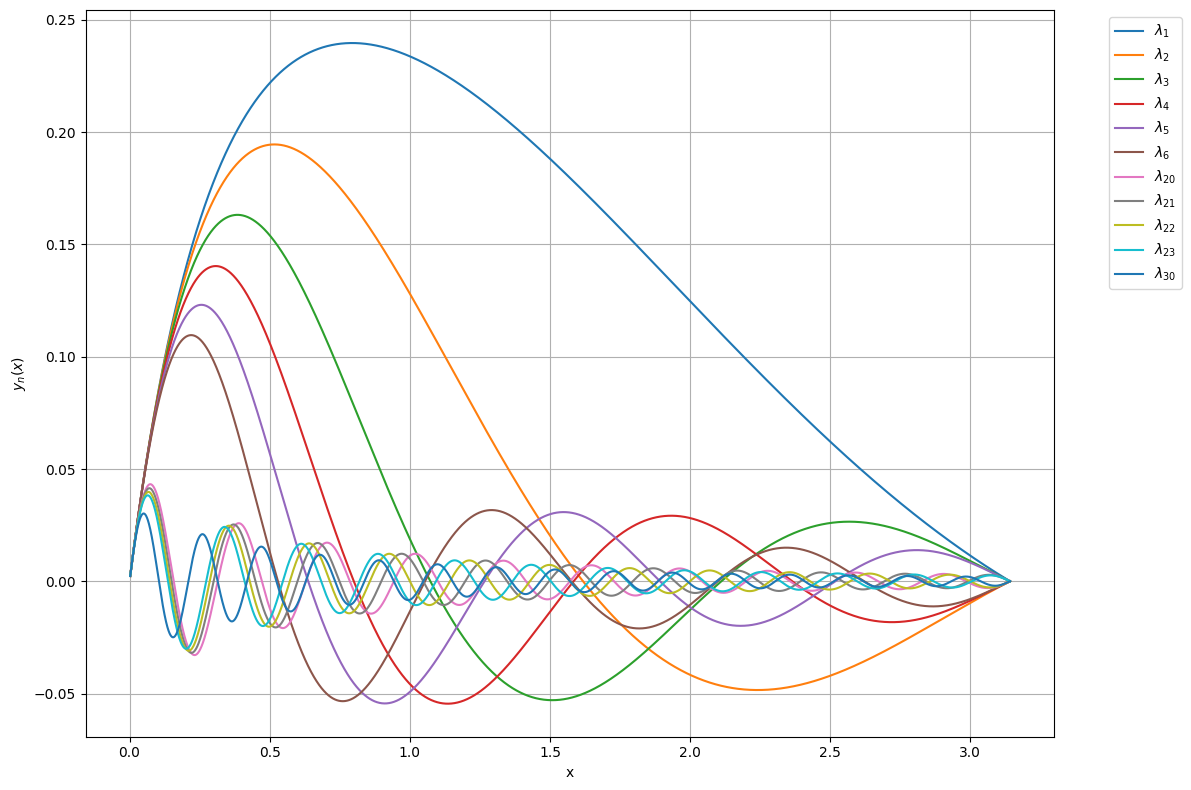}
    \caption{Eigenfunctions associated with the eigenvalues listed in Table~\ref{Tab:ejemplo1}, computed using the NSBF method.}
    \label{fig:placeholder}
\end{figure}\\
For computing the Weyl function we employ the analytical approximation $M_{\kappa,N}$ given by formula \eqref{eq:AproximationNSBFweyl}. 
In Figure \ref{fig:weylejemplo1} we observe the graph of the Weyl function associated with the Dirichlet problem.
For the real part of $\lambda$ we take values ranges from $-10$ to $10$ and for the imaginary part from $0$ to $10$. The observed jumps are located in the eigenvalues of the Dirichlet problem.
\begin{figure}[h!]
    \centering
    \includegraphics[width=1\linewidth]{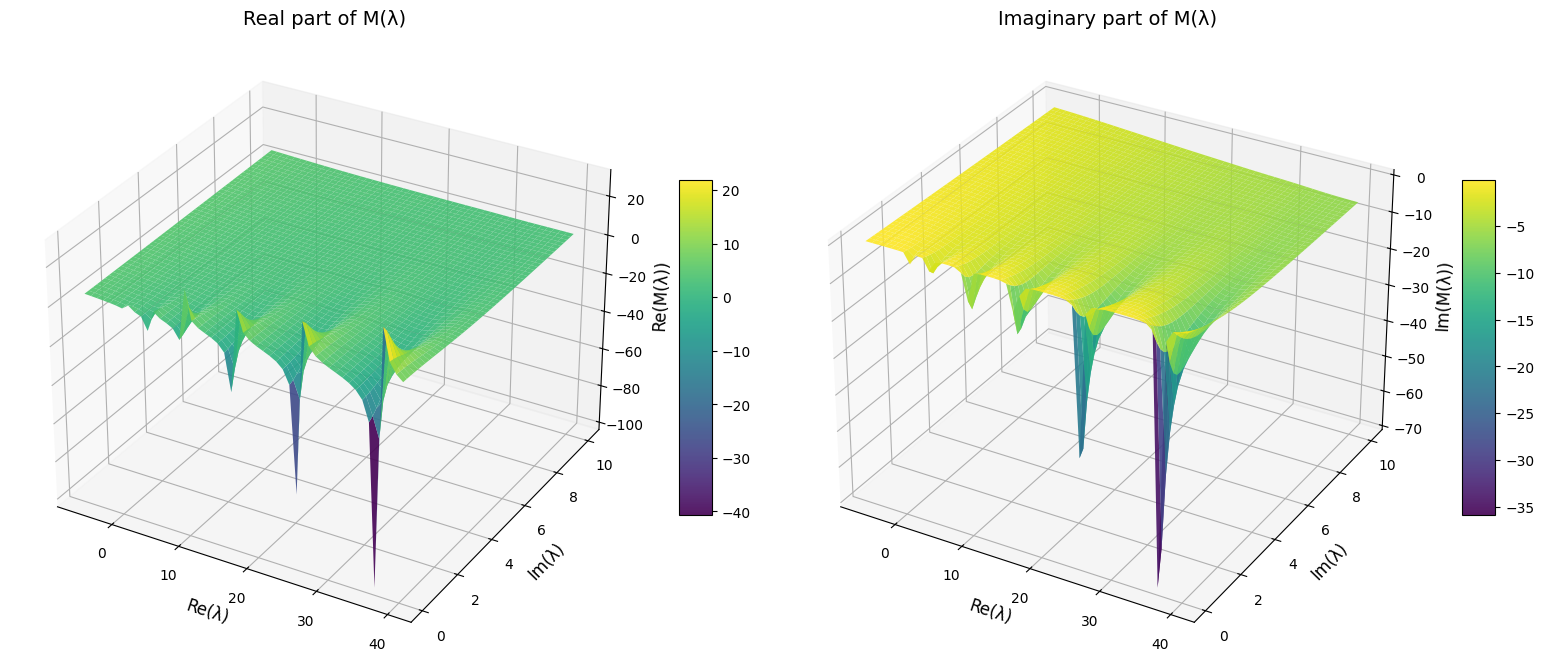}
    \caption{Graph of the real and imaginary part of Weyl function of the Dirichlet problem for impedance $a(x)=(1+x)^2$ in the $\lambda$ complex plane.}
    \label{fig:weylejemplo1}
\end{figure}

\end{example}

\FloatBarrier
\begin{example}[Non smooth impedance]
    In this example we consider the triangular conductivity function 
    \begin{equation*}
    \kappa(x)=
        \begin{cases}
            1+x,& \text{if  } \;0\leq x\leq \frac{1}{2},\\
            2-x,& \text{if  } \;\frac{1}{2}< x \leq 1.
        \end{cases}
    \end{equation*}
The graph of the triangular conductivity show Figure \ref{fig:ImpedanciaTriangulito}. In this case $L=1$ and $\kappa \in W^{1,2}(0,1)$, but its first derivative is a Heaviside-type function, so $\kappa \not\in W^{2,2}(0,1)$. Consequently, this equation cannot transform to a regular Schr\"odinger equation. The Dirichlet problem does not admit an explicit characteristic equation. We employ the NSBF series with $N=40$ and we computed $63$ eigenvalues.
    \begin{figure}[h!]
    \centering
    \includegraphics[width=.7\linewidth]{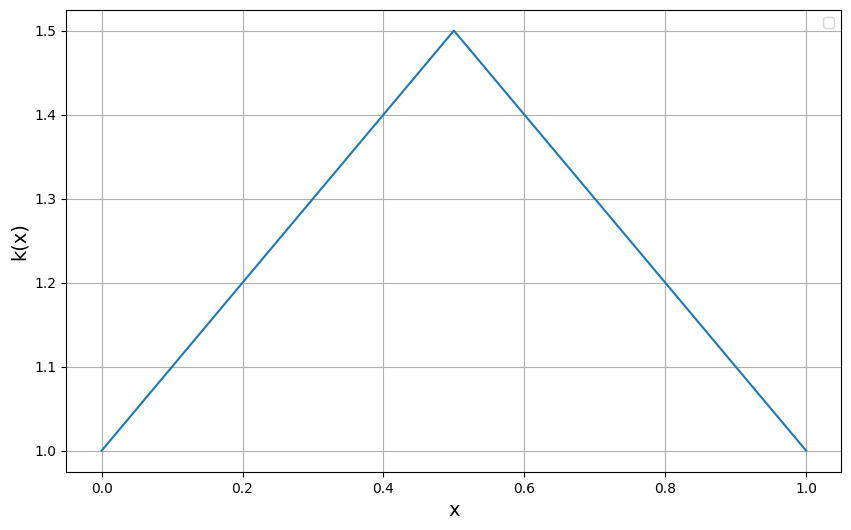}
    \caption{The graph of the triangular conductivity.}
    \label{fig:ImpedanciaTriangulito}
    \end{figure}\\
We compare the eigenvalues reported in Gao \cite{Gao}, which they computed using the Descent Flow method, with the ones that we obtained via the SPPS method and with our own results using NSBF.
For the SPPS method, we observe that for $n=11$ a significant error arises. This behavior is explained by the fact that SPPS relies on a Taylor series based representation, whose validity is inherently local. Consequently, it is common in the literature to apply the spectral shift technique \cite{spps} to improve the numerical performance of SPPS under such circumstances.
In the contrast, the NSBF method does not require any spectral shift, due to the uniform approximation of the eigenvalues ensured by the result stated in inequality \eqref{eq:estimateNSBF}.

\begin{table}[h!]\label{Tab:ejemplo2}
\centering
\begin{tabular}{cccc}
\toprule
$n$ & $\lambda_{\text{Gao}} \cite{Gao}$ & $\lambda_{\text{SPPS}} \cite{spps}$ & $\lambda_{\text{NSBF}}$ \\
\midrule
1  & 8.3513e+00   & 8.351443468115294      & 8.35158241342649   \\
2  & 3.9316e+01   & 39.31584966139984      & 39.3158490304397   \\
3  & 8.7324e+01   & 87.32512433822502      & 87.3252433641185   \\
4  & 1.5775e+02   & 157.7480942179849      & 157.748091718428   \\
5  & 2.4524e+02   & 245.2396600060401      & 245.239779974236   \\
6  & 3.5514e+02   & 355.1395825656164      & 355.139577098527   \\
7  & 4.8211e+02   & 482.1103986679944      & 482.110512751393   \\
8  & 6.3148e+02   & 631.4879392265606      & 631.487939226560  \\
9  & 7.9793e+02   & 798.7713794569980      & 798.771379456998   \\
10 & 9.8680e+02   & 985.2201158113872      & 985.220115811387  \\
11 & 1.1927e+03   & 1196.413213808432      & 1192.73369448276   \\
\bottomrule
\end{tabular}
\caption{Comparison of eigenvalues computed by Gao \cite{Gao}, SPPS \cite{spps}, and NSBF methods.}
\label{tab:eigenvalues_comparisonEj.2}
\end{table}

    \begin{figure}[h!]
    \centering
    \includegraphics[width=1\linewidth]{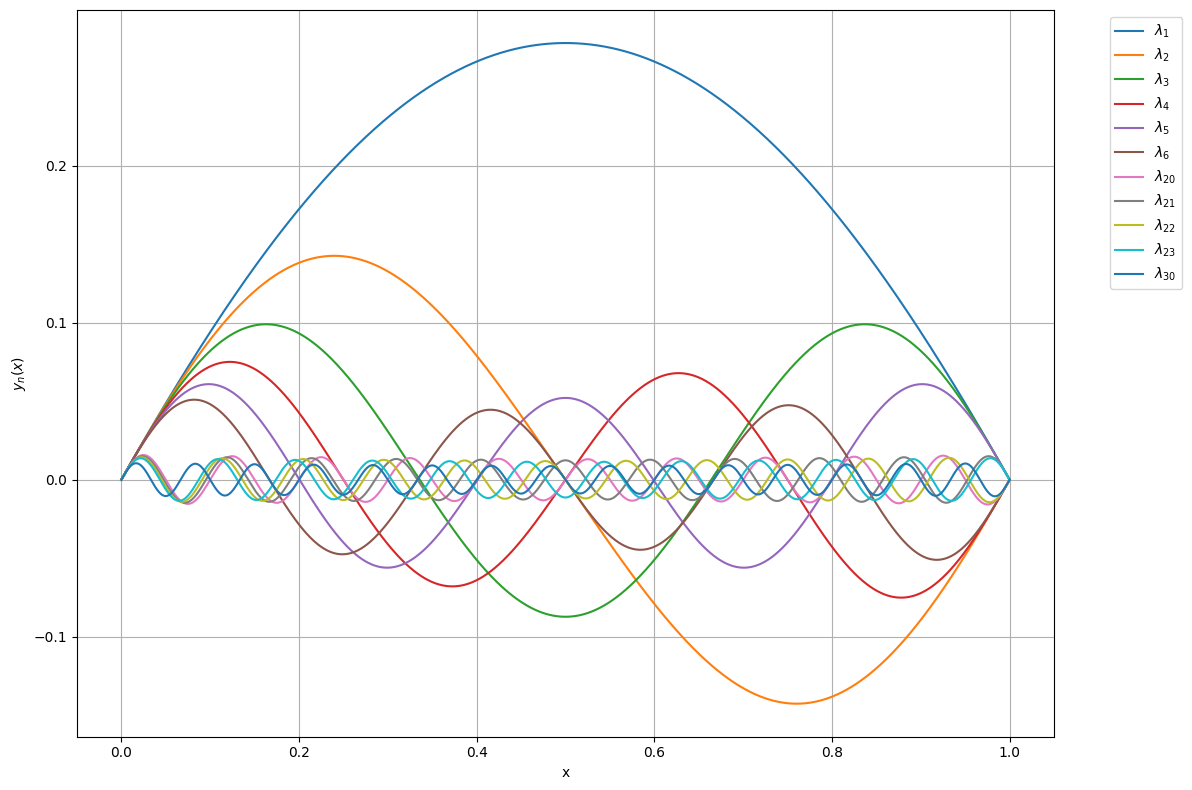}
    \caption{The graph of the eigenfunctions associated to the Dirichlet problem of the triangular conductivity.}
    \label{fig:EigenfuncionesTriangulito}
    \end{figure}
In Figure \ref{fig:WeylTriangulito} we observe the graph of the Weyl function associated with the Dirichlet problem.
For the real part of $\lambda$ we take values ranges from $0$ to $200$ and for the imaginary part from $0$ to $10$.
    \begin{figure}[h!]
    \centering
    \includegraphics[width=1\linewidth]{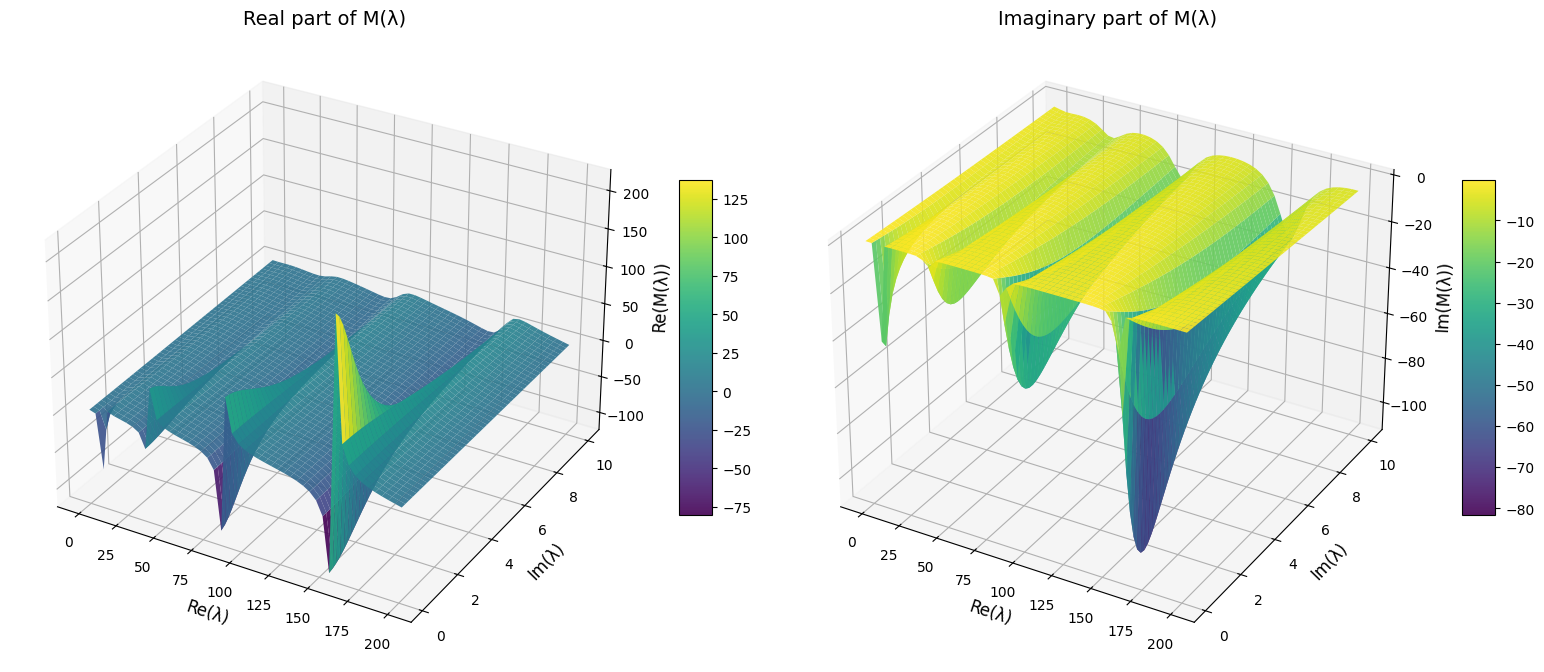}
    \caption{Graph of the real and imaginary part of Weyl function of the Dirichlet problem for the triangular conductivity.}
    \label{fig:WeylTriangulito}
    \end{figure}
\end{example}

\FloatBarrier
\section{Discussion and summary}
In this work, we derived a Neumann series of Bessel functions (NSBF) representation for the solution of Sturm-Liouville equation in impedance form. The obtained method provides a solid and general framework, applicable even when $\kappa \in W^{1,2}(0,L)$ and no additional smoothness is assumed. This places our results as a natural extension of previous NSBF-based representations, which typically require more regular coefficients (see \cite{NSBF1,NSBF2}). \\
 The expansion is obtained directly from the canonical transmutation kernel associated with the equation, and not through a reduction to the Schrödinger form, and it yields explicit truncation error estimates \eqref{eq:estimateNSBF} that remain uniform in the spectral parameter. We further developed an efficient recursive procedure for the computation of the NSBF coefficients formulated entirely within the impedance setting, allowing the representation to remain valid for conductivities of low regularity.

Similarly, the NSBF series presented here can be employed to solve problems associated with the Schrödinger equation via the Liouville transformation. Although NSBF series are available even in the case of Schrödinger equations with distributional potentials, there is no effective control over the decay of the coefficients nor over the truncation error of the NSBF series \cite{mineshrodinger}. One of the main advantages of NSBF series for impedance equations is the availability of a uniform bound of the form \eqref{eq:estimateNSBF} for both the truncation error and the decay of the coefficients. This leads to improved numerical stability in the computation of approximate eigenvalues and eigenfunctions. It is also worth noting that, in principle, NSBF representations can be implemented for complex-valued impedance functions. In this case, the computation of complex roots can be carried out using algorithms based on Rouché’s theorem and the argument principle, similar to those employed in \cite{sppsulises}.

From a numerical perspective, the proposed NSBF representation leads to an efficient and fully analytical scheme for the computation of solutions of the impedance equation. The recursive construction of the coefficients, together with the uniform truncation error control, allows for the accurate computation of eigenvalues with a large number of significant digits, even for eigenvalues of high index. This feature is particularly relevant in applications where higher spectral modes play a particularly relevant role, such as in acoustics, geophysics, and wave propagation models. 
In addition, the NSBF representation of the Weyl function provides a computationally accessible tool for the analysis of spectral data associated with the SLEIF. In particular, it can be employed for the validation and testing of mathematical models, as well as for the development and numerical implementation of inverse methods related to impedance-type Sturm–Liouville problems \cite{Rodcross}.
\newline 

 {\bf Acknowlegdments:}
 Abigail G. M\'arquez-Hern\'andez thanks to SECIHTI and Professor Vladislav V. Kravchenko for the financial support given through the scholarship \textit{Ayudante de Investigador Nacional SNII 3}.
 
	Víctor A. Vicente-Ben\'itez thanks to SECIHTI for their support through the program {\it Estancias Posdoctorales por México Convocatoria 2023 (I)}.

    Both autors thanks to Instituto de Matemáticas de la U.N.A.M. Unidad Querétaro (México), where part of this work was developed.
	\newline
	
	{\bf Conflict of interest:} This work does not have any conflict of interest.

\FloatBarrier


\begin{thebibliography}{99}
\bibitem{abramovitz} \textsc{M. Abramovitz, I. A. Stegun}, \textit{Handbook of mathematical functions}, New York:
Dover, 1972.

\bibitem{aktosun} \textsc{T. Aktosun, P. Sacks, X. Ch. Xu}, \textit{An inverse problem to determine the shape of a human vocal tract},  J. Comp. Appl. Math. 393, 113477 (2021).

\bibitem{almalki} \textsc{A. Almalki},\textit{ Sinc-galerkin method for Sturm–Liouville problem with applications in quantum mechanics}, J.Umm Al-Qura Univ. Appll. Sci. (2025)

\bibitem{benewitz} \textsc{C. Bennewitz, M. Brown, R. Weikard}, \textit{Spectral and Scattering Theory for Ordinary Differential Equations Vol. I: Sturm-Liouville Equations}, Springer, 2020.

\bibitem{sppscampos} \textsc{H. Blancarte, H. Campos, K. V. Khmelnytskaya} \textit{Spectral parameter power series for discontinuous coefficients}, Math. Meth. in the Appl. Sci. 38(10) (2015), 2000-2011.

\bibitem{carroll} \textsc{R. Carroll, F. Santosa}, \textit{Scattering Techniques for a One Dimensional Inverse Problem in Geophysics}, Math. Meth. in the Appl. Sci. 3 (1981) 145–171.

\bibitem{carroll1} \textsc{R. Carroll, F. Santosa, J. Ortega}, \textit{Stability for the one dimensional inverse problem via the Gelfand-Levitan equation}, Applicable Analysis, 13:4, (1982) 271-277.

\bibitem{carroll2} \textsc{R. Carroll, F. Santosa, L. Payne}, \textit{On complete recovery of geophysical data}, Math. Meth. Appl. Sci., 4, (1982) 33-73.

\bibitem{Everitt} \textsc{W. N. Everitt}, \textit{A catalogue of Sturm-Liouville differential equations, Theory, Past and Present}, Birkhüser, Basel, 2005, pp. 271-331.

\bibitem{Gao} \textsc{Q. Gao}, \textit{Decent flow methods for Sturm-Liouville problems}, Applied Mathematical Modelling 36 (2012), 4452-4465.

\bibitem{gladwell} \textsc{G. M. L. Gladwell}, \textit{Inverse Problems in Vibration}, 2nd ed., Kluver Academic, New York, 2005; Moscow, 2008.

\bibitem{Goldberg1991} \textsc{D. Goldberg}, \textit{What every computer scientist should know about floating-point arithmetic.} ACM Comput. Surv. 23(1) (1991) 5--48.

\bibitem{sppsnelson} \textsc{N. Guti\'errez-Jim\'enez, S. M. Torba},
		\textit{Spectral parameter power series representation for solutions of linear system of two first order differential equations}, Applied Mathematics and Computation,
		Vol. 370 (2020) 124911

 \bibitem{kravchenkospp1} \textsc{V. V. Kravchenko}, {\it A representation for solutions of the Sturm-Liouville equation}, Complex Variables and Elliptic Equations, 2008, v. 53, No. 8, 775-789.
 
\bibitem{NSBF1} \textsc{V. V. Kravchenko, L.J. Navarro, S.M. Torba}, \textit{Representation of solutions to the one-dimensional Schrödinger equation in terms of Neumann series of Bessel functions.} Appl. Math. Comput. 314(1) (2017) 173–192.

\bibitem{spps} \textsc{V. V. Kravchenko, R. M. Porter}, \textit{Spectral parameter power series for Sturm-Liouville problems}, Math. Methods Appl. Sci. 33 (2010) 459–468.

\bibitem{sppsordenn} \textsc{V.V. Kravchenko, R.M. Porter, S.M. Torba}, \textit{Spectral parameter power series for arbitrary order linear differential equations}, Math. Methods Appl. Sci. 42(15), (2019) 4902–4908. 

\bibitem{kravchenkodarboux} \textsc{V. V. Kravchenko, S. M. Torba}, \textit{Transmutations for Darboux transformed operators with applications}, J. Phys. A: Math. Theor., 45 (2012) 075201 (21 pp.).

\bibitem{NSBF2} \textsc{V.V. Kravchenko, S.M. Torba}, \textit{A Neumann series of Bessel functions representation for solutions of Sturm–Liouville equations}, Calcolo 55, 11 (2018).

\bibitem{Rodcross} \textsc{V. V. Kravchenko, S. M. Torba, A.O. Vatulyan}, \textit{Recovery of the rod cross-section shape}, Bol. Soc. Mat. Mex. 31, 46(2025). Doi: 10.1007/s40590-025-00727-7

\bibitem{sppsulises} \textsc{V.V. Kravchenko, S.M. Torba, U. Velasco-Garc\'ia}, \textit{Spectral parameter power series for Sturm–Liouville equations with a potential polynomially dependent on the spectral parameter and Zakharov-Shabat systems}. J. Math. Phys. 56, 073508 (2015).

\bibitem{mineimpedance1} \textsc{V. V. Kravchenko, V. A. Vicente-Ben\'{\i}tez}, \textit{Transmutation operators method for Sturm-Liouville equations in impedance form I: construction and analytical properties}, J Math Sci (2022) 266: 103-132. Doi: 10.1007/s10958-022-05875-z

\bibitem{nxky} \textsc{N. X. Ky}, \textit{On approximation of functions by polynomials with weight}, Acta. Math. Hung., 59(1-2) (1992), 49-58

\bibitem{matslice} \textsc{V. Ledoux, M. Van Daele, G. Vanden Berghe}, 
\textit{MATSLICE: A spectral solver for Sturm--Liouville equations}, 
ACM Trans. Math. Softw. 42 (2015), no. 2, Article~14, 1--18.

\bibitem{levin} \textsc{B. Ya. Levin, I. V. Ostrovskii}, \textit{Small perturbations of the set of roots of sine-type functions}, Izv. Akad. Nauk SSSR Ser. Mat. 43 (1979), no. 1, 87-110 (in Russian); Engl. transl.: Math. USSR-Izv. 14 (1979), no. 1, (1980) 79-101.

\bibitem{marchenko} \textsc{V. A. Marchenko}, \textit{Sturm-Liouville operators and applications}, Birkh\"auser, Basel, 1986.

\bibitem{miklavic} \textsc{M. Miklavcic}, \textit{Applied Functional Analysis and Partial Differential Equations}, World Scientific, Singapore, 1998.

\bibitem{pruvnikov} \textsc{A. P. Prudnikov, Yu. A. Brychkov, O. I. Marichev}, \textit{Integrals and Series. vol. 2. Special Functions} (Gordon \& Breach Science Publishers, New York, 1986)

\bibitem{pryce} \textsc{J. D. Pryce}, \textit{Numerical Solution of Sturm–Liouville Problems}, Clarendon Press, Oxford, 1993.

\bibitem{santosa} \textsc{F. Santosa}, \textit{Numerical scheme for the inversion of acoustical impedance profile based on the Gelfand-Levitan method}, Geophysical Journal International, Vol. 70, (1982), 229–243.

\bibitem{sirca} \textsc{Simon Širca, Martin Horvat}, \textit{Computational Methods in Physics, Compendium for Students}, Graduate Texts in Physics, Springer cham, 2025. 

\bibitem{suetin} \textsc{P. K. Suetin}, \textit{ Classical orthogonal polynomials}, 3rd ed. (in Russian), Moscow:
Fizmatlit, 2005, 480 pp.

\bibitem{vatulyan} \textsc{A. O. Vatulyan}, \textit{Inverse problems of solid mechanics}, Fizmatlit, Moscow, 2007 (in Russian).

\bibitem{mineimpedance3} \textsc{V. A. Vicente-Benítez}, \textit{Complete systems of solutions, transmutations and Darboux transform for Sturm-Liouville equations in impedance form
		},	Math Meth Appl Sci (2025), published online: http://doi.org/10.1002/mma.70328

\bibitem{mineshrodinger}  \textsc{V. A. Vicente-Benítez}, \textit{Transmutation operators for Schr\"odinger equations with distributional potentials and its related impedance equation}, arXiv:2511.12094        


\bibitem{vlahakis} \textsc{N. Vlahakis}, \textit{The Schwarzian Approach in Sturm–Liouville Problems}, Symmetry 2024, 16(6), 648.

\bibitem{webster} \textsc{A. G. Webster}, \textit{Acoustical impedance, and the theory of horns and of the phonograph}, Proc. Natl. Acad. Sci. U.S.A. 5 (1919), 275-282.



\bibitem{Wu} \textsc{Q. Wu, F. Fricke}, \textit{Determination of blocking locations and cross-sectional area in a duct by eigenfrequency shifts}, J . Acoustical SOC. 87 (1990) pp. 67-75.



	\end{thebibliography}
\end{document}